\documentclass[a4paper, 10pt, DIV10, headinclude=false, footinclude=false]{scrarticle}

\usepackage[utf8]{inputenc}
\usepackage[T1]{fontenc}
\usepackage[english]{babel}

\usepackage[colorlinks=false]{hyperref}

\usepackage{amsmath,amssymb,amsfonts,amsthm}
\usepackage{graphicx}
\usepackage{subcaption}
\usepackage{latexsym}
\usepackage{esint} %
\usepackage{aligned-overset}
\usepackage{multirow}
\usepackage{bigdelim}
\usepackage{empheq} 
\usepackage{enumitem}
\usepackage{tikz}
\usetikzlibrary{plotmarks}
\usepackage{tikzscale}
\usepackage{pgfplots}
\pgfplotsset{compat=newest} %
\usepackage{pgfplotstable}
\usetikzlibrary{positioning,shapes.multipart}
\usepgfplotslibrary{groupplots,dateplot}
\usetikzlibrary{patterns,shapes.arrows}

\usepackage{todonotes}
\usepackage{color}
\usepackage{xcolor}

\usepackage{caption}
\usepackage{algorithm}
\usepackage{algorithmic}

\usepackage{xparse}
\usepackage{xspace}
\usepackage{xstring}
\usepackage{stmaryrd}

\usepackage[noabbrev,capitalise]{cleveref}

\usepackage{amsopn}

\DeclareMathAlphabet\mathbfcal{OMS}{cmsy}{b}{n}

\newcommand{\bbN}{\mathbb{N}}

\newcommand{\pmat}[1]{\begin{pmatrix} #1 \end{pmatrix}}

\newcommand{\tn}[1]{t^{#1}}
\newcommand{\ts}{\tau}
\newcommand{\tauh}{\tfrac{\tau}{2}}

\newcommand{\tausf}{\tfrac{\tau^2}{4}}

\renewcommand{\dim}{d}
\newcommand{\dtau}{\partial_\tau}

\newcommand{\abs}[1]{\left\vert #1 \right\vert}
\newcommand{\ip}[1]{\big( #1 \big)}

\renewcommand{\d}[1]{\,\mathrm{d}#1}

\makeatletter
\def\mathcolor#1#{\@mathcolor{#1}}
\def\@mathcolor#1#2#3{%
  \protect\leavevmode
  \begingroup
    \color#1{#2}#3%
  \endgroup
}
\makeatother

\newtheorem{theorem}{Theorem}[section]
\newtheorem{lemma}[theorem]{Lemma}

\crefname{assumption}{assumption}{assumptions}
\Crefname{Assumption}{Assumption}{Assumptions}

\theoremstyle{remark}
\newtheorem{remark}[theorem]{Remark}

\theoremstyle{definition}
\newtheorem{definition}[theorem]{Definition} 

\numberwithin{equation}{section}

\NewDocumentCommand{\derv}{O{} m}{\partial_{#2}^{#1}}

\NewDocumentCommand{\projLtwo}{s}{
  \IfBooleanTF{#1}{\Pi_{\meshdiam}}{\pi_{\meshdiam}}
}

\NewDocumentCommand{\grhs}{o}{
  \IfNoValueTF{#1}{
    g
  }{
    g_{\sol[#1]}
  }
}

\NewDocumentCommand{\grhshn}{o m}{
  \IfNoValueTF{#1}{
    g_{\meshdiam}^{#2}
  }{
    g_{\sol[#1], \meshdiam}^{#2}
  }
}

\NewDocumentCommand{\avggrhshn}{o m}{
  \IfNoValueTF{#1}{
    \overline{g}_{\meshdiam}^{#2}
  }{
    \overline{g}_{\sol[#1], \meshdiam}^{#2}
  }
}

\NewDocumentCommand{\sol}{o}{
  \IfNoValueTF{#1}{x}{
    \IfEqCase{#1}{
      {first}{u}
      {second}{v}
      {dummy}{w}
    }
  }
}

\NewDocumentCommand{\soltest}{o}{
  \IfNoValueTF{#1}{y}{
    \IfEqCase{#1}{
      {first}{\psi}
      {second}{\phi}
    }
  }
}

\NewDocumentCommand{\exsol}{o m}{\widehat{\sol[#1]}^{#2}}

\NewDocumentCommand{\solh}{o m}{\sol[#1]_{\meshdiam}^{#2}}

\NewDocumentCommand{\avgsolh}{o m}{\overline{\sol[#1]}_{\meshdiam}^{#2}}

\NewDocumentCommand{\soltesth}{o}{\soltest[#1]_{\meshdiam}}

\NewDocumentCommand{\material}{o}{
  \IfValueTF{#1}
    {\mathcal{M}_{\sol[#1]}}
    {\mathcal{M}}
}

\NewDocumentCommand{\friedop}{o}{
  \IfValueTF{#1}
    {\mathcal{L}_{\sol[#1]}}
    {\mathcal{L}}
}

\NewDocumentCommand{\friedopadj}{o}{
  \IfValueTF{#1}
    {\mathcal{L}_{\sol[#1]}^{\ast}}
    {\mathcal{L}^{\ast}}
}

\NewDocumentCommand{\friedcoeff}{o m}{
  \IfValueTF{#1}
    {L_{#2, \sol[#1]}}
    {L_{#2}}
}

\NewDocumentCommand{\friedoph}{o}{
  \IfValueTF{#1}
    {\mathcal{L}_{\sol[#1],{\meshdiam}}}
    {\mathcal{L}_{\meshdiam}}
}

\NewDocumentCommand{\friedophleapfrog}{o}{
  \friedoph[#1]^{\mathrm{lf}}
}

\NewDocumentCommand{\friedophmodified}{o}{
  \friedoph[#1]^{m}
}

\NewDocumentCommand{\friedbndop}{o}{
  \IfValueTF{#1}
    {\mathcal{L}_{\boundary, \sol[#1]}}
    {\mathcal{L}_{\boundary}}
}

\NewDocumentCommand{\friedbndfield}{o}{
  \IfValueTF{#1}
    {L_{\boundary, \sol[#1]}}
    {L_{\boundary}}
}

\NewDocumentCommand{\friedassbndop}{o}{
  \IfValueTF{#1}
    {\mathcal{L}_{\partial, \sol[#1]}}
    {\mathcal{L}_{\partial}}
}

\NewDocumentCommand{\friedassbndfield}{o m}{
  \IfValueTF{#1}
    {L_{\partial, \sol[#1]}^{#2}}
    {L_{\partial}^{#2}}
}

\NewDocumentCommand{\Id}{o}{
    \mathcal{I}_{\IfNoValueTF{#1}{}{\sol[#1]}}
}

\newcommand{\filfun}{\Psi}

\newcommand{\filfunhat}{\widehat{\filfun}}

\newcommand{\auxfilfun}{\Theta}

\newcommand{\auxfilfuncardi}{\varphi}

\NewDocumentCommand{\friedopsecondorder}{s o}{
  \IfBooleanTF{#1}{
    \mathcal{A}_{\IfValueT{#2}{\mathrm{#2}}}
  }{
    \friedoph[second] 
    \IfValueT{#2}{
      \IfEqCase{#2}{
        {lf}{\cutoffLeapfrog}
        {m}{\cutoffModified}
      }
    } \friedoph[first]
  }
}

\NewDocumentCommand{\filfunophat}{s o}{
  \IfBooleanTF{#1}
    {\filfunhat(-\tau^2\friedopsecondorder[#2])}
    {\mathbf{\filfunhat}}
}

\NewDocumentCommand{\filfunop}{s o}{
  \IfBooleanTF{#1}
    {\filfun(-\tau^2\friedopsecondorder[#2])}
    {\mathbf{\filfun}}
}

\NewDocumentCommand{\auxfilfunop}{s o}{
  \IfBooleanTF{#1}
    {\auxfilfun(-\tau^2\friedopsecondorder[#2])}
    {\mathbf{\auxfilfun}}
}

\newcommand{\RA}{\auxfilfunop_{\Id}}

\NewDocumentCommand{\perturbationop}{m o}{
  \IfValueTF{#2}
    {\mathbfcal{D}_{\mathrm{#1},\sol[#2]}}
    {\mathbfcal{D}_{\mathrm{#1}}}
}

\NewDocumentCommand{\auxfilfuncardiop}{s o}{
  \IfBooleanTF{#1}
    {\auxfilfuncardi(-\tau^2\friedopsecondorder[#2])}
    {\boldsymbol{\auxfilfuncardi}}
}

\NewDocumentCommand{\genop}{s o o}{
  \IfBooleanTF{#1}{\widehat{\mathbfcal{R}}}{\mathbfcal{R}}
  \IfValueT{#2}{_{%
    \IfEqCase{#2}{%
      {p}{+}%
      {m}{-}%
      {pm}{\pm}%
    }%
    \IfValueT{#3}{%
    \IfEqCase{#3}{%
        {CN}{,\mathrm{CN}}%
    }%
    }%
	}%
  	}%
	\IfValueT{#3}{^{%
	\IfEqCase{#3}{%
		{inv}{-1}%
	}}%
	}%
}

\NewDocumentCommand{\dofs}{o}{
  \IfValueTF{#1}
    {N_{\meshdiam, \sol[#1]}}
    {N_{\meshdiam}}
}

\NewDocumentCommand{\massmatrix}{o}{
  \IfValueTF{#1}
    {\mathbf{M}_{\sol[#1], \meshdiam}}
    {\mathbf{M}_{\meshdiam}}
}

\NewDocumentCommand{\friedmatrix}{o}{
  \IfValueTF{#1}
    {\mathbf{L}_{\sol[#1], \meshdiam}}
    {\mathbf{L}_{\meshdiam}}
}

\NewDocumentCommand{\friedmatrixmodified}{o}{
  \friedmatrix[#1]^{m}
}

\NewDocumentCommand{\solvector}{o m}{
  \mathbf{\sol[#1]}^{#2}
}

\NewDocumentCommand{\rhsvector}{o m}{
  \mathbf{g}^{#2}
}

\NewDocumentCommand{\scdordermatrixmodified}{s}{
  \IfBooleanTF{#1}
    {(\massmatrix[first]^m)^{-1}\friedmatrixmodified[second](\massmatrix[second]^m)^{-1}\friedmatrixmodified[first]}
    {\mathbf{A}_{m}}
}

\NewDocumentCommand{\scdordermatrixmodifiedmodified}{s}{
  \IfBooleanTF{#1}
    {\massmatrix[second]^{-1}\friedmatrixmodified[first]\massmatrix[first]^{-1}\friedmatrixmodified[second]}
    {\widehat{\mathbf{A}}_{m}}
}

\NewDocumentCommand{\defectTrapez}{o m}{\delta_{\mathrm{tr} \IfValueT{#1}{,\sol[#1]}}^{#2}}

\NewDocumentCommand{\err}{m o}{e_{\IfValueT{#2}{\sol[#2]}}^{#1}}

\NewDocumentCommand{\discerr}{m o}{e_{{\meshdiam} \IfValueT{#2}{, \sol[#2]}}^{#1}}

\NewDocumentCommand{\projerr}{o m}{e_{\pi\IfValueT{#1}{, \sol[#1]}}^{#2}}

\NewDocumentCommand{\projerrwithop}{o m O{}}{e_{\pi, \friedop[#1]#2}^{#3}}

\NewDocumentCommand{\defect}{o m o}{\delta_{\meshdiam \IfValueT{#1}{,\mathrm{#1}}\IfValueT{#3}{,\sol[#3]}}^{#2}}

\NewDocumentCommand{\defectCN}{o m}{\delta_{\meshdiam, \mathrm{CN} \IfValueT{#1}{,\sol[#1]}}^{#2}}

\NewDocumentCommand{\defectLF}{o m}{\delta_{\meshdiam, \mathrm{lf} \IfValueT{#1}{,\sol[#1]}}^{#2}}

\NewDocumentCommand{\defectLTI}{o m}{\delta_{\meshdiam, \mathrm{LTI} \IfValueT{#1}{,\sol[#1]}}^{#2}}

\NewDocumentCommand{\defectsmooth}{o m}{\delta_{\meshdiam \IfValueT{#1}{,\sol[#1]}, s}^{#2}}

\NewDocumentCommand{\defectfried}{o m}{\delta_{\meshdiam \IfValueT{#1}{,\sol[#1]}, \friedoph}^{#2}}

\NewDocumentCommand{\ipDiff}{o m}{\Delta_{\mathrm{m}\IfValueT{#1}{, \sol[#1]}}(#2)}

\NewDocumentCommand{\bilDiff}{o m m}{\Delta_{{#2}\IfValueT{#1}{, \sol[#1]}}(#3)}

\NewDocumentCommand{\polydegmost}{O{}}{\mathbb{Q}_{\dim}^{#1}}

\NewDocumentCommand{\polytotaldegmost}{O{}}{\mathbb{P}_{\dim}^{#1}}

\NewDocumentCommand{\normal}{O{} m}{\mathbf{n}^{#2}_{#1}}

\NewDocumentCommand{\problemdim}{o}{m_{\IfValueT{#1}{\sol[#1]}}}

\newcommand{\domain}{\Omega}
\newcommand{\boundary}{\Gamma}

\NewDocumentCommand{\meshdiam}{O{}}{h^{#1}}

\NewDocumentCommand{\maxmeshdiam}{O{}}{h_{\mathrm{max}}^{#1}}

\NewDocumentCommand{\minmeshdiam}{O{}}{h_{\mathrm{min}}^{#1}}

\NewDocumentCommand{\meshdiamCoarse}{O{}}{\meshdiam[#1]_c}

\NewDocumentCommand{\meshdiamFine}{O{}}{\meshdiam[#1]_f}

\NewDocumentCommand{\maxmeshdiamCoarse}{O{}}{\meshdiam[#1]_{c,\mathrm{max}}}

\NewDocumentCommand{\minmeshdiamCoarse}{O{}}{\meshdiam[#1]_{c,\mathrm{min}}}

\NewDocumentCommand{\maxmeshdiamFine}{O{}}{\meshdiam[#1]_{f,\mathrm{max}}}

\NewDocumentCommand{\minmeshdiamFine}{O{}}{\meshdiam[#1]_{f,\mathrm{min}}}

\newcommand{\mesh}{\mathcal{T}_{\meshdiam}}
\newcommand{\cell}{K}
\newcommand{\cellfine}{\cell_f}

\newcommand{\cutoffModified}{\chi_m}

\newcommand{\cutoffLeapfrog}{\chi_{\mathrm{lf}}}

\newcommand{\meshCoarse}{\mathcal{T}_{\meshdiam,c}}
\newcommand{\meshFine}{\mathcal{T}_{\meshdiam,f}}

\newcommand{\meshLeapfrog}{\mathcal{T}_{\meshdiam,\mathrm{lf}}}

\newcommand{\meshModified}{\mathcal{T}_{\meshdiam,m}}

\NewDocumentCommand{\norm}{m O{}}{\Vert #1 \Vert_{#2}}

\NewDocumentCommand{\seminorm}{m O{}}{\vert #1 \vert_{#2}}

\NewDocumentCommand{\dgip}{o m}{\ip{#2}_{\domain}}%

\NewDocumentCommand{\dgnorm}{o m}{\norm{#2}[\domain]}%

\NewDocumentCommand{\opnorm}{o m}{\norm{#2}[\dgspace[#1]\leftarrow\dgspace[#1]]}

\NewDocumentCommand{\normHkTh}{m m}{\norm{#2}[#1, \mesh]}

\NewDocumentCommand{\seminormHkTh}{m m}{\seminorm{#2}[#1, \mesh]}

\NewDocumentCommand{\seminormHkcell}{m m}{\seminorm{#2}[#1, \cell]}

\NewDocumentCommand{\normLoneLtwo}{m}{\norm{#1}[L^1([\tn{0}, T], \Ltwo[\domain])]}

\newcommand{\real}{\mathbb{R}}

\newcommand{\realvec}[1]{{\real}^{#1}}

\NewDocumentCommand{\frieddom}{o}{
  \mathcal{D}(\friedop[#1])
}

\NewDocumentCommand{\graphspace}{o}{H(\friedop[#1])}

\NewDocumentCommand{\dgspace}{o}{\mathbb{V}_{\meshdiam}\IfValueT{#1}{^{\sol[#1]}}}

\NewDocumentCommand{\brokenpoly}{s o m}{\polydegmost[#3](\mesh)\IfBooleanF{#1}{^{\problemdim[#2]}}}

\NewDocumentCommand{\sobolev}{m m}{H^{#1}(#2)}

\NewDocumentCommand{\brokensobolev}{s o m}{\sobolev{#3}{\mesh}\IfBooleanF{#1}{^{\problemdim[#2]}}}

\NewDocumentCommand{\dgfrieddom}{o}{V_{\meshdiam}^{\friedop[#1]}}

\NewDocumentCommand{\Ltwo}{o}{L^2\IfValueT{#1}{(#1)}}

\NewDocumentCommand{\Ltwovec}{o m}{L^2(#2)^{\problemdim[#1]}}

\NewDocumentCommand{\Linfty}{o}{L^{\infty}\IfValueT{#1}{(#1)}}

\NewDocumentCommand{\Linftymatrix}{o m}{L^{\infty}(#2)^{\problemdim[#1] \times \problemdim[#1]}}

\NewDocumentCommand{\sobolevinftymatrix}{o m m}{W^{#2, \infty}(#3)^{\problemdim[#1] \times \problemdim[#1]}}

\NewDocumentCommand{\brokensobolevinftymatrix}{o m}{\sobolevinftymatrix[#1]{#2}{\mesh}}

\newcommand{\dGdegree}{k}

\newcommand{\lfcdegree}{p}
\newcommand{\lfcstabparam}{\eta}
\newcommand*{\Pp}[1][\!\lfcdegree]{P_{#1}}

\newcommand*{\Tp}[1][\lfcdegree]{T_{\!#1}}

\newcommand*{\nup}{\nu_\lfcdegree}

\newcommand*{\alphap}[1][\lfcdegree]{\alpha_{#1}}

\newcommand*{\hbeta}[1][\filfun]{\beta_{#1}}

\NewDocumentCommand{\CFriedop}{o}{C_{\friedop[#1]}}

\NewDocumentCommand{\CFriedInvIneq}{o}{C_{\mathrm{inv}, \friedop \IfValueT{#1}{,\sol[#1]}}}

\NewDocumentCommand{\CFriedInvIneqCoarse}{o}{C_{\mathrm{inv}, \friedop, \IfValueT{#1}{\sol[#1],} c}}

\NewDocumentCommand{\Capprox}{o}{C_{\pi, \friedop \IfValueT{#1}{,\sol[#1]}}}

\newcommand{\LTSStabConst}{C_{stb,\, c}}

\newcommand{\Cproj}{C_{\pi}}

\newcommand{\cAuxfilfunmin}{c_{\auxfilfun}}
\newcommand{\cAuxfilfuncardi}{C_\auxfilfuncardi}
\newcommand{\cAuxfilfuncardimax}{\widehat{C}_\auxfilfuncardi}
\newcommand{\cAmin}{\cAuxfilfunmin}

\newcommand{\cflThetaLeapfrogCoarse}{\vartheta_c}

\newcommand{\cflThetaLeapfrogAll}{\vartheta}

\newcommand*{\tsCFLLFAll}{\ts_\mathrm{CFL, lf}}
\newcommand*{\tsCFLLFcoarse}{\ts_\mathrm{CFL, lf, c}}
\newcommand*{\tsCFLfilfun}{\ts_{\mathrm{CFL},\filfun}}

\newcommand{\exE}{E}
\newcommand{\exH}{H}
\newcommand{\exJ}{J}

\newcommand{\permitt}{\varepsilon}
\newcommand{\permeab}{\mu}

\newcommand{\cf}{cf.\xspace}
\newcommand{\ie}{i.e.\xspace}

\NewDocumentCommand{\dG}{s}{dG\IfBooleanF{#1}{\xspace}}

\NewDocumentCommand{\wellposed}{s}{well-posed\IfBooleanF{#1}{\xspace}}

\NewDocumentCommand{\Friedrichs}{s}{Friedrichs'\IfBooleanF{#1}{\xspace}}

\NewDocumentCommand{\Friedrichstitle}{s}{FRIEDRICHS'\IfBooleanF{#1}{\xspace}}

\NewDocumentCommand{\Maxwells}{s}{Maxwells\IfBooleanF{#1}{\xspace}}

\NewDocumentCommand{\discontinuous}{s}{discontinuous\IfBooleanF{#1}{\xspace}}

\NewDocumentCommand{\Galerkin}{s}{Galerkin\IfBooleanF{#1}{\xspace}}

\NewDocumentCommand{\Lts}{s}{Local time-stepping\IfBooleanF{#1}{\xspace}}

\NewDocumentCommand{\Ltstitle}{s}{LOCAL TIME-INTEGRATION\IfBooleanF{#1}{\xspace}}

\NewDocumentCommand{\lts}{s}{local time-stepping\IfBooleanF{#1}{\xspace}}

\NewDocumentCommand{\lti}{s}{local time-integration\IfBooleanF{#1}{\xspace}}

\NewDocumentCommand{\Lti}{s}{Local time-integration\IfBooleanF{#1}{\xspace}}

\NewDocumentCommand{\righthandside}{s}{right-hand-side\IfBooleanF{#1}{\xspace}}

\NewDocumentCommand{\leapfrog}{s}{leapfrog\IfBooleanF{#1}{\xspace}}

\NewDocumentCommand{\Leapfrog}{s}{Leapfrog\IfBooleanF{#1}{\xspace}}

\NewDocumentCommand{\CN}{s}{Crank-Nicolson\IfBooleanF{#1}{\xspace}}

\NewDocumentCommand{\LFC}{s}{Leapfrog-Chebychev\IfBooleanF{#1}{\xspace}}

\NewDocumentCommand{\lfc}{s}{leapfrog-Chebychev\IfBooleanF{#1}{\xspace}}

\NewDocumentCommand{\CFL}{s}{CFL\IfBooleanF{#1}{\xspace}}

\NewDocumentCommand{\selfadj}{s}{self-adjoint\IfBooleanF{#1}{\xspace}}

\NewDocumentCommand{\pd}{s}{positive definite\IfBooleanF{#1}{\xspace}}

\NewDocumentCommand{\psd}{s}{positive semi-definite\IfBooleanF{#1}{\xspace}}

\NewDocumentCommand{\wrt}{s}{w.r.t.\IfBooleanF{#1}{\xspace}}

\NewDocumentCommand{\timestepsize}{s}{time stepsize\IfBooleanF{#1}{\xspace}}

\newenvironment{abstr}[1]{ \vspace{.05in}\footnotesize
	\parindent .2in
	{\upshape\bfseries #1. }\ignorespaces}{\par\vspace{.1in}}
	\newenvironment{Abstract}{\begin{abstr}{Abstract}}{\end{abstr}}
	\newenvironment{keywords}{\begin{abstr}{Key words}}{\end{abstr}}
	\newenvironment{AMS}{\begin{abstr}{AMS subject classifications}}{\end{abstr}}
	\makeatletter
\newcommand{\mylabel}[2]{#2\def\@currentlabel{#2}\label{#1}}
\makeatother

\allowdisplaybreaks[4]

\begin{document}

\title{
  Local time-integration for Friedrichs' systems%
  \thanks{Funded by the Deutsche Forschungsgemeinschaft (DFG, German Research Foundation) – Project-ID 258734477 – SFB 1173}
}

\author{
  Marlis Hochbruck\thanks{Institute for Applied and Numerical Mathematics, Karlsruhe Institute for Technology, 76131 Karlsruhe, Germany}
  \and Malik Scheifinger\footnotemark[2]
}
\date{}
\maketitle

\begin{Abstract}
In this paper, we address the full discretization of Friedrichs' systems with a two-field structure, such as Maxwell's equations or the acoustic wave equation in div-grad form, cf.~\cite{MFO2023}. We focus on a discontinuous Galerkin space discretization applied to a locally refined mesh or a small region with high wave speed. This results in a stiff system of ordinary differential equations, where the stiffness is mainly caused by a small region of the spatial mesh.
When using explicit time-integration schemes, the time step size is severely restricted by a few spatial elements, leading to a loss of efficiency. As a remedy, we propose and analyze a general leapfrog-based scheme which is motivated by~\cite{CarH22}. The new, fully explicit, local time-integration method filters the stiff part of the system in such a way that its CFL condition is significantly weaker than that of the leapfrog scheme while its computational cost is only slightly larger. For this scheme, the filter function is a suitably scaled and shifted Chebyshev polynomial. 
While our main interest is in explicit local-time stepping schemes, the filter functions can be much more general, for instance, a certain rational function
leads to the locally implicit method, proposed and analyzed in~\cite{HocS16}. Our analysis provides sufficient conditions on the filter function to ensure full order of convergence in space and second order in time for the whole class of local time-integration schemes.
\end{Abstract}

\begin{keywords}
time integration, 
space discretization,
Friedrichs' system,
wave-type problems,
local time-stepping methods, 
locally implicit methods,
stability analysis, 
error analysis, 
discontinuous Galerkin
\end{keywords}

\begin{AMS}
65M12, 65M15, 65M22
\end{AMS}

\section{Introduction}\label{sec:intro}

The aim of this paper is to construct and analyze a new local time-integration (LTI) method for the discontinuous Galerkin (dG) space discretization of \Friedrichs systems in a two-field structure. 
Using an explicit time-integrator like the \leapfrog method on the full spatial domain necessitates satisfying a Courant–Friedrichs–Lewy (\CFL) condition, meaning that the \timestepsize $\ts$ is proportional to the inverse of the smallest diameter within the spatial mesh. 
On the other hand, implicit methods like the \CN method are unconditionally stable, but require the solution of large linear systems of equations. 
We are interested in problems, where the \CFL condition is dominated by a small number of mesh elements, i.e., tiny elements or elements with a high wave speed. 
In this situation, standard explicit or implicit methods are inefficient. The basic idea of LTI methods is to use an explicit time-integrator, for example the \leapfrog method, on the large, nonstiff part of the system and to couple it to a tailored method on the small, stiff part. The modified method can either be an explicit method with a weaker CFL condition (resulting in a \lts (LTS) method) or an unconditionally stable implicit scheme (leading to a locally-implicit (LI) method). Overall one aims at using a method with a \CFL condition which only depends on the nonstiff part of the system. Since the higher computational cost of the modified method arises only on a small part of the degrees of freedom, the construction leads to a problem adapted and efficient local time-integration scheme.

In recent years, many LTI schemes were introduced for different wave-type equations. LTS methods for the acoustic wave equation in second-order formulation were constructed in \cite{AlmM17,CarH24,DiaG09,DiaG15,GroMM15,GroMS21,GroM15,GroM13}, for \Maxwells equations in second-order formulation in~\cite{GroM10,KotFL20,MonPFC08,Pip06} and LI methods for \Maxwells equations in~\cite{DesLM13,DesLM14,DesLM17,DolFFL10,HocS16,HocS19,BonR02,Ver11}, for instance. Moreover, in~\cite{MFO2023} the LI method from~\cite{HocS16} was generalized to \Friedrichs systems.

However, to the best of our knowledge fully explicit LTS methods for \Friedrichs systems have not been considered so far. 
Here, we construct and analyze a class of \leapfrog-based LTI schemes 
motivated by~\cite{CarH24}. 
We will show that the method is stable under a \CFL condition which is independent of the stiff part of the problem and is of second-order in time and optimal order in space.
As a special case, one variant of the LTI scheme corresponds to the LI method from~\cite{MFO2023} and the theoretical results reduce to the ones given there.

\paragraph{Outline}
The paper is structured as follows. First, in \Cref{sec:friedrichs}, we briefly review \Friedrichs systems and introduce the notation which we will use throughout the paper.

Our main results are contained in \Cref{sec:lts-scheme}. Here, we start with an introduction of the splitting of the discrete operators into a stiff and a nonstiff part. Then, we present the new LTI scheme based on the \leapfrog method with a general filter function and study its stability. Important examples leading to LTS and LI methods are discussed in detail. 
For the whole class of methods, we prove stability if a \timestepsize restriction based on properties of the filter function is satisfied. Moreover, we prove that LTI methods satisfy error bounds with optimal order in space and second-order in time. 

At last, in \Cref{sec:numerical-examples}, we present numerical experiments which substantiate our theoretical results for \Maxwells equations. Moreover, we show that the LTS method outperforms the LI and the \leapfrog method on these test problems.

\section{Friedrichs' systems}  \label{sec:friedrichs}
Let \(\domain\subset\realvec{\dim}\) be an open, polygonal domain and denote its boundary by \(\boundary = \partial\domain\). %

Following~\cite[\S 11.2]{MFO2023} we consider the two-field structured \Friedrichs system
\begin{subequations}\label{eq:two-field-system}
    \begin{align}
        \derv{t}\sol[first] &= \friedop[second]\sol[second] + \grhs[first], &&\text{on } \domain\times\real_+, \\
        \derv{t}\sol[second] &= \friedop[first]\sol[first] + \grhs[second],  &&\text{on } \domain\times\real_+, \\
        \sol[first](0) &= \sol[first]^0, \quad \sol[second](0) = \sol[second]^0, &&\text{on } \domain \nonumber
    \end{align}
\end{subequations}
with initial values \(\sol[first]^0, \sol[second]^0\) and right-hand-side \(\grhs\).
The \Friedrichs operators \(\friedop[first]\) and \(\friedop[second]\) are given as
\begin{equation}\label{eq:two-field-friedops}
    \material[second]\friedop[first]\sol[first] = \sum_{i = 1}^{\dim} \friedcoeff{i}\partial_i\sol[first], 
    \qquad
    \material[first]\friedop[second]\sol[second] = \sum_{i = 1}^{\dim} \friedcoeff{i}^{\mathrm{T}}\partial_i\sol[second],
\end{equation}
with \(\friedcoeff{i}\in\real^{\problemdim[second]\times\problemdim[first]}\) for \(i = 0,\dots,\dim\). 
The material tensors \(\material[first]\in\Linftymatrix[first]{\domain}\), \(\material[second]\in\Linftymatrix[second]{\domain}\) are assumed to be symmetric and positive definite, and \(\problemdim[first], \problemdim[second]\in\bbN\) denote the number of components of \(\sol[first]\) and \(\sol[second]\).
The boundary conditions of~\eqref{eq:two-field-system} are embodied into the domains \(\frieddom[first]\) and \(\frieddom[second]\), see~\cite{DiPE12,MFO2023} for details.
We further assume that the boundary conditions do not introduce damping, \cf~\cite[Assump. 11.8]{MFO2023}. 
This leads to the important adjointness property
\begin{equation}\label{eq:friedop-adjointness}
    \ip{\friedop[first]\sol[first], \sol[second]} 
    = 
    - \ip{ \sol[first],\friedop[second]\sol[second]}
    \qquad\text{for all } 
    \sol[first]\in\frieddom[first], \sol[second]\in\frieddom[second] .
\end{equation}
Special cases are linear \Maxwells equations or the acoustic wave-equation in div-grad formulation, for details see~\cite[\S 7.1]{DiPE12} or~\cite[\S 9.3]{MFO2023}.

\subsection{Notation}\label{sec:notation}
Let \(\cell\subseteq\realvec{\dim}\) be open. For two vector-valued functions \(\sol[first], \tilde{\sol[first]}\in\Ltwovec[first]{\cell}\) we denote the \(\Ltwo\)-inner product by
\begin{equation*}
    \ip{\sol[first], \tilde{\sol[first]}}_{\Ltwo[\cell]} = \int_\cell \sol[first]\cdot\tilde{\sol[first]}\d{x} 
\end{equation*}
and the \(\Ltwo\)-inner product weighted by \(\material[first]\) as
\begin{align*}
  \ip{\sol[first], \Tilde{\sol[first]}}_{\cell} 
  = \ip{\sol[first], \Tilde{\sol[first]}}_{\material[first], \cell} 
    = \ip{
        \material[first]\sol[first], 
        \Tilde{\sol[first]}
    }_{\Ltwo[\cell]},%
\end{align*}
where we supress the indices \(\material[first]\) whenever possible. The inner product weighted by \(\material[second]\) is defined analogously. From  the arguments of the inner product it is clear which weight is used.
The induced norms are denoted accordingly.

Furthermore, we denote by \(\partial_i\) the \(i\)th weak derivative, \(i = 1,\dots,\dim\). Higher derivatives are denoted by \(\partial^\alpha = \partial_1^{\alpha_1}\dots\partial_d^{\alpha_d}\) for a multi index \(\alpha\in\bbN_0^d\), and we write \(\abs{\alpha} = \alpha_1 + \dots + \alpha_d\).

The Sobolev spaces in \(\Ltwo(\cell)\) are denoted as \(\sobolev{\ell}{\cell}\), \(\ell\in\bbN_0\), and are equipped with the weighted norms
\begin{equation*}
  \norm{\sol[first]}[\ell, \cell]^2
  =
    \norm{\sol[first]}[\material[first], \ell, \cell]^2
    =
    \sum_{j=0}^{\ell} \abs{\sol[first]}_{\material[first], j, \cell}^2,
    \qquad
    \abs{\sol[first]}_{j, \cell}^2    
    =\abs{\sol[first]}_{\material[first], j, \cell}^2
    =
    \sum_{\abs{\alpha} = j} \norm{\partial^\alpha\sol[first]}[\material[first], \cell]^2.
\end{equation*}
Again we suppress the index \(\material[first]\) whenever possible. 

In what follows, we recall some notation and results from~\cite{DiPE12, MFO2023}. We denote by \(\mesh\) meshes of \(\domain\) which we assume to be admissible~\cite[Def. 1.57]{DiPE12}. By \(\meshdiam_\cell\) we denote the diameter of a cell \(\cell\in\mesh\). We further define \(\meshdiam\in\Linfty(\domain)\) by \(\meshdiam\vert_\cell = \meshdiam_\cell\) and \(\maxmeshdiam = \max_{\cell} \meshdiam_\cell\) as well as \(\minmeshdiam = \min_{\cell} \meshdiam_\cell\).

For each component of~\eqref{eq:two-field-system} we denote by 
\begin{align*}
    \dgspace[first] 
    &= 
    \{\sol[first]\in\Ltwo[\domain] \ \vert \ \sol[first]\vert_\cell\in\polytotaldegmost[\dGdegree](\cell) \text{ for all } \cell\in\mesh\}^{\problemdim[first]}, \\
    \dgspace[second] 
    &= 
    \{\sol[second]\in\Ltwo[\domain] \ \vert \ \sol[second]\vert_\cell\in\polytotaldegmost[\dGdegree](\cell) \text{ for all } \cell\in\mesh\}^{\problemdim[second]}, 
\end{align*}
the spaces of broken polynomials of total degree at most \(\dGdegree\) and the product space by \(\dgspace = \dgspace[first] \times \dgspace[second]\).

Furthermore, we symbolize the \(\Ltwo\)-projection \wrt the weighted inner product \(\ip{\cdot, \cdot}_{\domain}\) by \(\projLtwo\colon\Ltwo[\domain]\to\dgspace\), see~\cite{MFO2023} for details.

At last, we recall the broken Sobolev spaces
\begin{equation*}
    \brokensobolev*[first]{\ell} 
    = 
    \{
        \sol[first]\in\Ltwo[\domain]
        \ \vert \ 
        \sol[first]\vert_\cell\in\sobolev{\ell}{\cell}\text{ for all } \cell\in\mesh
    \}.
\end{equation*}
with weighted norms
\begin{equation*}
    \normHkTh{\ell}{\sol[first]}^2 = \sum_{j=0}^{\ell} \seminormHkTh{j}{\sol[first]}^2, 
    \qquad 
    \seminormHkTh{j}{\sol[first]}^2 = \sum_{\cell\in\mesh}\seminormHkcell{j}{\sol[first]}^2.
\end{equation*}

\section{Local time-integration}\label{sec:lts-scheme}

After a central-fluxes \dG discretization~\cite{MFO2023,ErnG21,HocK20}, the spatially discrete wave-type problem~\eqref{eq:two-field-system} takes the form
\begin{subequations}\label{eq:two-field-dg-system}
    \begin{align}
        \derv{t}\solh[first]{} &= \friedoph[second]\solh[second]{} + \grhshn[first]{}, &\text{on } \real_+, \\
        \derv{t}\solh[second]{} &= \friedoph[first]\solh[first]{} + \grhshn[second]{},  &\text{on } \real_+, \\
        \solh[first]{}(0) &= \projLtwo\sol[first]^0, \quad \solh[second]{}(0) = \projLtwo\sol[second]^0, \quad  \grhshn{} = \projLtwo\grhs, &
    \end{align}
\end{subequations}
where the discrete operators \(\friedoph[first], \friedoph[second]\) inherit the adjointness property~\eqref{eq:friedop-adjointness}, \ie,
\begin{equation}\label{eq:adjointness}
    \ip{\friedoph[first]\solh[first]{}, \solh[second]{}}_{\domain} 
    = -\ip{\solh[first]{}, \friedoph[second]\solh[second]{}}_{\domain}
    \qquad 
    \text{for all }\solh[first]{}\in\dgspace[first],\ \solh[second]{}\in\dgspace[second] .
\end{equation}
Before we present our \lti scheme, we follow~\cite{HocS16} and split the mesh into two parts
\begin{equation}
    \mesh = \meshFine\; \dot{\cup}\; \meshCoarse.
\end{equation}
The set \(\meshFine\) contains all elements with a small diameter or a high wave speed and the set \(\meshCoarse\) all others. We further define  
\begin{align}
    \minmeshdiamCoarse = \min_{\cell\in\meshCoarse} \meshdiam_\cell, 
    \qquad 
    \minmeshdiamFine = \min_{\cell\in\meshFine} \meshdiam_\cell .
\end{align}
Our approach necessitates an additional division of the mesh such that the set \(\meshModified\) contains also the neighbors of elements in \(\meshFine\) and \(\meshLeapfrog\) only contains coarse elements with coarse neighbors, \ie,  
\begin{align*}
    \meshModified &= \{\cell\in\mesh\ \colon \ \exists \cellfine\in\meshFine~\text{s.t.}~ \cell, \cellfine \text{ share a face}\}, \\ %
    \meshLeapfrog &= \mesh\setminus\meshModified .
\end{align*}
This one extra layer of coarse elements in \(\meshModified\) is essential to the independence of small mesh diameters or high wave speeds in the CFL condition as we will prove later on.

We further define the cutoff-operators
\begin{equation}\label{lts:eq:cutoff}
    \cutoffModified\sol[first] = 
    \begin{cases}
        \sol[first],\ \text{on}\ \cell\in\meshModified, \\
        0,\ \text{on}\ \cell\in\meshLeapfrog,
    \end{cases}
    \quad\text{and}\quad
    \cutoffLeapfrog\sol[first] = 
    \begin{cases}
        0,\ \text{on}\ \cell\in\meshModified, \\
        \sol[first],\ \text{on}\ \cell\in\meshLeapfrog .
    \end{cases}
\end{equation}
Now we are able to propose our \lti method. With a smooth function \(\filfun\colon [0, \infty)\to\real\) satisfying \(\filfun(0) = 1\), the fully discrete \lti scheme is defined as
\begin{subequations}\label{eq:lts-scheme}
    \begin{align}
        \avgsolh[second]{n+1/2} - \solh[second]{n} &= \tauh\friedoph[first]\solh[first]{n} + \tauh\grhshn[second]{n+1/2}, \label{eq:lts-scheme-a} \\
        \solh[first]{n+1} - \solh[first]{n} &= \tau\filfunop\left(\friedoph[second]\avgsolh[second]{n+1/2} + \grhshn[first]{n+1/2}\right), \label{eq:lts-scheme-b} \\
        \solh[second]{n+1} - \avgsolh[second]{n+1/2} &= \tauh\friedoph[first]\solh[first]{n+1} + \tauh\grhshn[second]{n+1/2}, \label{eq:lts-scheme-c} \\
        \filfunop &= \filfunop*[m]\ . \label{eq:lts-scheme-d}
    \end{align}
\end{subequations}
Note that in~\cite{CarH24} we used the notation \(\filfunhat\) instead of \(\filfun\).

We denote by 
\begin{equation}\label{eq:avg-rhs-def}
	\grhshn{n+1/2} = \tfrac12(\grhshn{n+1} + \grhshn{n})
\end{equation}
the average of $\grhshn{}$ at two consecutive time steps. Note that the filter function \(\filfunop\) and hence the modified method only work on \(\meshModified\) plus coupling and that for \(\filfun \equiv 1\), the scheme reduces to the \leapfrog scheme.

The remaining part of this paper is devoted to the proof of the following full-discretization error estimate.
\begin{theorem}\label{Thm:error-estimate-shortened}
	For \(\solh{n} = \pmat{\solh[first]{n} & \solh[second]{n}}\) given by~\eqref{eq:lts-scheme} with \(\tau\) satisfying a \CFL condition independent of \(\minmeshdiamFine\) and a sufficiently regular solution \(\sol = \pmat{\sol[first] & \sol[second]}\) of~\eqref{eq:two-field-system} we obtain
	\begin{align}
		\dgnorm{\sol(\tn{n}) - \solh{n}} 
		&\leq 
        C(\tau^2 + \maxmeshdiam[\dGdegree])
	\end{align}
	with a constant \(C\) independent of \(\tau\) and \(\maxmeshdiam\).
\end{theorem}
Details on the \CFL condition will be given later.

\subsection{Stability}\label{subsec:stability}

In the following, we assume that \(\filfunop\) defined in \eqref{eq:lts-scheme-d} is invertible, which is indeed the case under a certain restriction on the time stepsize  as we will later see. We start by rewriting~\eqref{eq:lts-scheme} into a system inspired by the analysis of the \CN method~\cite{HocS16} imitating a discrete semi-group technique. To do so, we define 
\begin{subequations}
\begin{align}\label{eq:lts-R-operators}
	\genop[pm]
	&= 
	\pmat{
		\auxfilfunop & \pm\tauh\friedoph[second] \\ 
		\pm\tauh\friedoph[first] & \Id[second]
	}
	-
	\tausf\pmat{
		-\friedopsecondorder[lf] & 0 \\ 
		0 & 0
	}
\end{align}
with
\begin{align}\label{eq:lts-thetafun}
	\auxfilfunop &= \auxfilfunop*[m],\\
	\label{eq:auxfilfun}
	\auxfilfun(z) 
	&= \filfun(z)^{-1} \bigl(1 - \frac{z}{4} \filfun(z) \bigr)
	= \frac{1}{\filfun(z)} - \frac{z}{4},
	\qquad
	\auxfilfun(0) = 1 .
\end{align}
\end{subequations}
In~\Cref{tab:methods-filfuns-vals}, we collect special choices of $\filfun$ and $\auxfilfun$ together with certain parameters, which we will introduce later for the analysis. More details will be given in Section~\ref{subsec:special-filfuns} below.
\begin{lemma}\label{Lem:lts-scheme-transf}
	The approximations \(\solh{n} = \pmat{\solh[first]{n} & \!\! \solh[second]{n}}\) of the \lti\ scheme \eqref{eq:lts-scheme}
	satisfy
	\begin{equation}\label{eq:lts-scheme-transf}
		\genop[m]\solh{n+1} = \genop[p]\solh{n} + \tau\grhshn{n+1/2}, 
		\qquad 
		\grhshn{n+1/2} = \pmat{\grhshn[first]{n+1/2} \\ \grhshn[second]{n+1/2}}.
	\end{equation}
\end{lemma}
A key idea is  to split $\genop[pm]$ into a positive definite operator $\RA$, a skew-adjoint operator 
$\friedoph$, and a self-adjoint, positive semidefinite perturbation operator $\perturbationop{lf}$ via
\begin{subequations}\label{eq:genop-splitting-all}
\begin{equation}\label{eq:genop-splitting}
     \genop[pm]= \RA \pm \tauh\friedoph - \tausf\perturbationop{lf} ,
\end{equation}
where
\begin{equation}\label{eq:perturb-op}
    \RA = \pmat{
		\auxfilfunop & 0 \\
		0            & \Id[second]
	},
    \quad
    \friedoph = \pmat{
        0 & \friedoph[second] \\
        \friedoph[first] & 0
    },
    \quad
	\perturbationop{lf} = \pmat{
		-\friedopsecondorder[lf] & 0 \\ 
		0                      & 0
	} .
\end{equation}
\end{subequations}
Here, the adjointness condition~\eqref{eq:adjointness} directly implies that \(\friedoph\) is skew-adjoint, \(\perturbationop{lf}\) and $\auxfilfunop$ are self-adjoint, and that \(\perturbationop{lf}\) is positive semidefinite.
However, the definiteness of $\RA$ is only guaranteed for special choices of $\auxfilfun$ (see Table~\ref{tab:methods-filfuns-vals}) or under additional conditions on its argument $-\ts^2\friedopsecondorder[m]$ in \eqref{eq:lts-thetafun}. To ensure this, we introduce additional notation. 
\begin{definition} \label{def:ctheta}
	Let \(\filfun\colon [0, \infty)\to\real\) be a smooth function satisfying \(\filfun(0) = 1\) and consider $\auxfilfun$ defined in \eqref{eq:auxfilfun}. 
    For a constant $\cAuxfilfunmin \in (0,1]$ we define $\hbeta = \hbeta(\cAuxfilfunmin) \in (0,\infty]$ as the maximum value such that
    \begin{equation} \label{eq:ctheta}
        0 < \filfun(z) \leq 1 
        \quad 
        \text{and}
        \quad
        \auxfilfun(z) \geq \cAuxfilfunmin
        \qquad
        \text{for all } z \in [0,\hbeta^2] \cap \real,
    \end{equation}
    and $\hbeta=\infty$, if \eqref{eq:ctheta} holds for all $z\geq0$.
\end{definition}
\begin{table}
\setlength{\tabcolsep}{5pt}
\begin{center}
    \begin{tabular}{|c|c c c|c c c|}
  		\cline{2-7}
		\multicolumn{1}{c|}{\rule{0pt}{4mm}} & $\filfun$ & $\auxfilfun$ & $\auxfilfuncardi$ & $\hbeta^2$ & $\cAuxfilfunmin$ & $\cAuxfilfuncardi$ \\[0.7ex]
        \hline%
        \rule{0pt}{4mm}\Leapfrog & $1$ & $1 - \tfrac{z}{4}$ & $-\tfrac14$ & $4 (1-\cAuxfilfunmin)$ %
        & $1 - \cflThetaLeapfrogAll^2$ & $\tfrac14$ \\[0.7ex]
        \hline
        \rule{0pt}{4mm}\CN & $ (1 + \tfrac{z}{4})^{-1}$ & $1$ & $0$ & $\infty$ & $1$ & $0$ \\[0.7ex]
        \hline
        \rule{0pt}{4mm}\LFC & \eqref{eq:lfc-polynomials} & \eqref{eq:auxfilfun} & \eqref{eq:auxfilfuncardi} & \eqref{eq:lfc-beta} & \eqref{eq:lfc-filfunBound-constants} & \eqref{eq:lfc-beta} \\[0.7ex]
        \hline
    \end{tabular}
\end{center}
\caption{Examples for filter functions with corresponding constants.}
\label{tab:methods-filfuns-vals}
\end{table}
The parameters $\hbeta^2$ and $\cAuxfilfunmin$ directly enter the  \CFL conditions to ensure stability of the scheme.
\begin{definition}\label{def:lts-CFL}
	For given \(\cflThetaLeapfrogAll, \cflThetaLeapfrogCoarse \in (0, 1]\) we define
	\begin{subequations}\label{eq:lfc-cfl}
	\begin{align}
		\tsCFLfilfun^2 
		&= 
		\frac{\hbeta^2}{\norm{\friedopsecondorder[m]}}, \label{eq:lfc-cfl-a} \\
		\tsCFLLFcoarse^2 
		&= 
		\tsCFLLFcoarse^2(\cflThetaLeapfrogCoarse) 
		=
		\frac{4\cAmin\cflThetaLeapfrogCoarse^2}{\norm{\friedopsecondorder[lf]}}, \label{eq:lfc-cfl-b}
	\end{align}		
	\end{subequations}
	and
	\begin{align}
		\tsCFLLFAll^2 
		= 
		\tsCFLLFAll^2(\cflThetaLeapfrogAll) = \frac{4\cflThetaLeapfrogAll^2}{\norm{\friedopsecondorder}}. \label{eq:leapfrog-cfl}
	\end{align}
\end{definition}
Recall that the \leapfrog scheme on the whole spatial domain is stable if $\tau \leq \tsCFLLFAll$.
Obviously, we aim at choosing $\filfun$ in such a way that the \CFL\ condition of the \lti method is significantly relaxed compared to the \leapfrog scheme.

For the dependence of the \CFL condition on the submeshes and the material parameters, note that
\begin{equation*}
	\norm{\friedopsecondorder[lf]}\lesssim\minmeshdiamCoarse[-2], 
	\qquad
	\norm{\friedopsecondorder[m]}\lesssim\minmeshdiamFine[-2] ,
\end{equation*}
with constants that depend on the largest eigenvalue of \(\material[first]^{-1}\friedcoeff{i}^{\mathrm{T}}\material[second]^{-1}\friedcoeff{i}\), \(i = 1,\dots,\dim\), on \(\meshCoarse\) and \(\meshFine\) respectively, \cf~\cite[Eq. 6]{HocK22} or~\cite[Lem. 7.32]{DiPE12}. Roughly speaking, the largest eigenvalues scale like the inverse of the product of the smallest eigenvalues of $\material[first]$ and $\material[second]$ on the respective submeshes. 
It is important to note that the eigenvalues of \(\friedopsecondorder[lf]\) and hence \(\tsCFLLFcoarse\) are independent of the material parameters and the diameters within the fine mesh \(\meshFine\), \cf~\cite{MFO2023,HocS16} for details.
\begin{lemma} \label{lts:Cor:A-op}
	For $\tau \leq \tsCFLfilfun$ defined~\eqref{eq:lfc-cfl-a}, the operator $\auxfilfunop$ induces a norm \(\norm{\cdot}[\auxfilfunop]\) such that
	for all $\solh[first]{} \in \dgspace[first]$ we have
    \begin{equation} \label{lts:eq:auxfilfun-norm}
    	        \cAuxfilfunmin \dgnorm[first]{\solh[first]{}}^2
    	\leq
        \norm{\solh[first]{}}[\auxfilfunop]^2 
            =
            \dgip[first]{\auxfilfunop \solh[first]{},\solh[first]{}} \,.
    \end{equation} 
\end{lemma}
\begin{proof}
	The statement follows immediately from~\eqref{eq:ctheta} since we can bound the eigenvalues $z$ of \(-\tau^2\friedopsecondorder[m]\) by \(\hbeta^2\) if \(\tau\leq\tsCFLfilfun\).
\end{proof}
\Cref{lts:Cor:A-op} implies that $\RA$ defined in~\eqref{eq:lts-thetafun} is also \selfadj and \pd for $\tau \leq \tsCFLfilfun$ and hence it induces a norm \(\norm{\cdot}[\RA]\) on $\dgspace$.
 
Using both \CFL conditions in~\eqref{eq:lfc-cfl} we can bound the action of the operators \(\genop[pm]\) defined in~\eqref{eq:lts-R-operators} from below.
\begin{lemma}\label{Cor:Rpm-lower-bound}
    Let $\cflThetaLeapfrogCoarse \in (0,1)$ and $\tau\leq\min\{\tsCFLLFcoarse(\cflThetaLeapfrogCoarse), \tsCFLfilfun\}$ satisfy the \CFL conditions~\eqref{eq:lfc-cfl}. Then we have
    \begin{equation}\label{eq:Rpm-lower-bound}
        \cAmin(1 - \cflThetaLeapfrogCoarse^2)\dgnorm[first]{\solh[first]{}}^2 + \dgnorm[second]{\solh[second]{}}^2 \leq \dgip{\genop[pm]\solh{}, \solh{}} .
    \end{equation}
	for all \(\solh{} = \pmat{\solh[first]{} & \solh[second]{}} \in \dgspace\).
\end{lemma}
\begin{proof}
	Let \(\solh{} = \pmat{\solh[first]{} & \solh[second]{}} \in \dgspace\). 
    Using~\eqref{eq:perturb-op} and \eqref{eq:lfc-cfl-b} we obtain
	\begin{align*}
		\tausf\dgip{\perturbationop{lf}\solh{}, \solh{}}
		= 
		-\tausf\dgip{\friedopsecondorder[lf]\solh[first]{}, \solh[first]{}}
		\leq 
		\cAmin\cflThetaLeapfrogCoarse^2\dgnorm[first]{\solh[first]{}}^2,
	\end{align*}
	and thus with~\eqref{eq:genop-splitting} and~\eqref{lts:eq:auxfilfun-norm}
	\begin{align*}
		\dgip{\genop[pm]\solh{},\solh{}} 
		&= 
		\dgip{\auxfilfunop\solh[first]{}, \solh[first]{}} 
		+ 
		\dgnorm[second]{\solh[second]{}}^2
		- 
		\tausf\dgip{\perturbationop{lf}\solh{}, \solh{}} 
		 \\
		&\geq
		\cAmin(1 - \cflThetaLeapfrogCoarse^2)\dgnorm[first]{\solh[first]{}}^2 
		+ 
		\dgnorm[second]{\solh[second]{}}^2.
	\end{align*}
    This proves the statement.
\end{proof}

In the next \Cref{Lemma:Rpm-bound-lf-all} we show
that the CFL condition for the \lti scheme~\eqref{eq:lts-scheme} can not become stronger than that of the \leapfrog scheme used on the whole grid.
\begin{lemma}\label{Lemma:Rpm-bound-lf-all}
	Let  \(\cflThetaLeapfrogAll \in (0,\, 1)\), \(\solh{} = \pmat{\solh[first]{} & \solh[second]{}}\in\dgspace\), and $\tau \leq \tsCFLLFAll(\cflThetaLeapfrogAll)$ defined in~\eqref{eq:leapfrog-cfl}. Moreover, assume that \(0 < \filfun(z) \leq 1\) holds for all \(z\in[0,\, 4]\). Then we have 	
	\begin{equation}\label{eq:Rpm-bound-lf-all}
		\dgip{\genop[pm]\solh{},\solh{}}
		\geq 
		(1 - \cflThetaLeapfrogAll^2)\dgnorm[first]{\solh[first]{}}^2 
		+ 
		\dgnorm[second]{\solh[second]{}}^2.
	\end{equation}
\end{lemma}
\begin{proof}
By assumption, $\filfunop$ is invertible since the \leapfrog \CFL condition~\eqref{eq:leapfrog-cfl} implies \(\dgnorm[first]{-\tau^2\friedopsecondorder }\leq 4\cflThetaLeapfrogAll^2 \leq 4\) and thus $\dgip{\filfunop^{-1}\solh[first]{}, \solh[first]{}} \geq \dgnorm[first]{\solh[first]{}}^2$.
Using \eqref{eq:auxfilfun} we can write
\begin{align*}
	\auxfilfunop + \tausf\friedopsecondorder[lf]
	& = \filfunop^{-1} + \tausf \bigl( \friedopsecondorder[m] + \friedopsecondorder[lf]\bigr) \\
    & = \filfunop^{-1} + \tausf  \friedopsecondorder.
\end{align*}
From \eqref{eq:lts-R-operators} and the \CFL condition \eqref{eq:leapfrog-cfl} we conclude 
	\begin{align*}
		\dgip{\genop[pm]\solh{},\solh{}}		
		&= 
		\dgip{(\auxfilfunop + \tausf\friedopsecondorder[lf])\solh[first]{},\solh[first]{}} + \dgnorm[second]{\solh[second]{}}^2 \\
		&\geq 
		(1 - \cflThetaLeapfrogAll^2)\dgnorm[first]{\solh[first]{}}^2 + \dgnorm[second]{\solh[second]{}}^2.
	\end{align*}
	This completes the proof.
\end{proof}
\Cref{Cor:Rpm-lower-bound} and the splitting of $\genop[pm]$ in~\eqref{eq:lts-R-operators} lead to the following properties.
\begin{lemma} \label{lts:lem:Rpm-properties}
	Let $\tau \leq \min\{\tsCFLfilfun, \tsCFLLFcoarse\}$ defined in~\eqref{eq:lfc-cfl} for some \(\cflThetaLeapfrogCoarse \in (0,\, 1)\).
	Then, for $\solh{}$, $\soltesth \in \dgspace$ and $\genop = \genop[m][inv] \genop[p]$ we have for $n=0,1,2,\ldots$ the identities
	\begin{subequations} \label{eq:lts-Rpm-properties}
		\begin{align}
			\dgip{\genop[m] \solh{},\soltesth} 
			&= \dgip{\solh{},\genop[p] \soltesth},
			\label{eq:lts-Rpm-properties-a}\\
			\dgip{\genop[m] \genop^n \solh{},\genop^n\solh{}} 
			&= \dgip{\genop[m] \solh{}, \solh{}},
			\label{eq:lts-Rpm-properties-b}\\
			\dgip{\genop[m] \genop^n \genop[m][inv] \solh{}, \genop^n \genop[m][inv]\solh{}} 
			&= \dgip{\genop[m][inv] \solh{},\solh{}},
			\label{eq:lts-Rpm-properties-bb}
		\end{align}
	\end{subequations}
	and in addition the bounds
	\begin{subequations}  \label{eq:lts-Rpm-bounds}
		\begin{align}  
		   \label{eq:lts-Rpm-bounds-a}
		   \dgnorm{\genop^n \solh{}} 
		   &\leq \LTSStabConst\norm{\solh{}}[\RA], \\
		   \label{eq:lts-Rpm-bounds-b}
		   \dgnorm{\genop^n \genop[m][inv] \solh{}} 
		   &\leq \LTSStabConst\dgnorm{\solh{}}, \\
		   \label{eq:lts-Rpm-inv-bounds}
		   \dgnorm{\genop[pm][inv] \solh{}}
		   &\leq \dgnorm{\solh{}},
		\end{align}
	\end{subequations}
	with \(\LTSStabConst = (\cAmin(1 - \cflThetaLeapfrogCoarse^2))^{-1/2}\).
\end{lemma}
\begin{proof}
	The adjointness property~\eqref{eq:lts-Rpm-properties-a} follows from the skew-adjointness of $\friedoph$ and the self-adjointness of \(\friedoph[second]\cutoffLeapfrog\friedoph[first]\), while~\eqref{eq:lts-Rpm-properties-b} can be seen by induction with
	\begin{align*}
		\dgip{\genop[m]\genop\solh{}, \genop\solh{}} 
		&= \dgip{\genop[p]\solh{}, \genop\solh{}} \\
		&= \dgip{\solh{}, \genop[m]\genop\solh{}} \\
		&= \dgip{\solh{}, \genop[p]\solh{}} \\
		&= \dgip{\genop[m]\solh{}, \solh{}}.
	\end{align*}
	One can easily verify that the inverses are given by
	\begin{align} \label{eq:lts-Rpmgeninv}
		\genop[pm][inv]
		&=
		\begin{pmatrix}
			\filfunop & \quad \mp \tauh \filfunop\friedoph[second] \\
			\mp \tauh \friedoph[first]\filfunop & \quad \Id[second] + \tausf \friedoph[first]\filfunop\friedoph[second]
		\end{pmatrix} \nonumber \\
		&=
		\begin{pmatrix}
			\filfunop & \\
			& \quad \Id[second] + \tausf \friedoph[first]\filfunop\friedoph[second]
		\end{pmatrix}
		\mp 
		\tauh
		\begin{pmatrix}
			& \filfunop\friedoph[second] \\
			\friedoph[first]\filfunop & 
		\end{pmatrix}.
	\end{align}
	These are decompositions of \(\genop[pm][inv]\) into their symmetric and skew-symmetric parts.

	Replacing $\solh{}$ by $ \genop[m][inv] \solh{} $ in \eqref{eq:lts-Rpm-properties-b} proves \eqref{eq:lts-Rpm-properties-bb}.
	Furthermore, \eqref{eq:Rpm-lower-bound} and \eqref{eq:lts-Rpm-properties-b}
	imply
	\begin{align*}
		\cAmin(1 - \cflThetaLeapfrogCoarse^2)\dgnorm[first]{\genop^n \solh{}}^2
		&\leq
		\dgip{\genop[m] \genop^n\solh{}, \genop^n\solh{}} \\
		&= 
		\dgip{\genop[m] \solh{}, \solh{}} \\
		&=
		\norm{\solh{}}[\RA]^2 
		- 
		\dgip{\perturbationop{lf}\solh{}, \solh{}} 
		\leq \norm{\solh{}}[\RA]^2,
	\end{align*}
	where we have used the positive semi-definiteness of \(\perturbationop{lf}\) in the last estimate. This proves \eqref{eq:lts-Rpm-bounds-a}. 

	By~\eqref{eq:ctheta}, \(\filfunop\) is positive definite if the \CFL condition~\eqref{eq:lfc-cfl-a} is satisfied. Thus we have with the adjointness property~\eqref{eq:adjointness}
	\begin{equation*}
		\dgip{\friedoph[first]\filfunop\friedoph[second]\solh[second]{}, \solh[second]{}}
		=
		-\dgip{\filfunop\friedoph[second]\solh[second]{}, \friedoph[second]\solh[second]{}}
		\leq 0
	\end{equation*}
	which shows
	\begin{align*}
		\dgip{\genop[pm][inv]\solh{}, \solh{}}
		&=
		\dgip{\filfunop\solh[first]{}, \solh[first]{}}
		+
		\dgip{\solh[second]{}, \solh[second]{}}
		+
		\tausf\dgip{\friedoph[first]\filfunop\friedoph[second]\solh[second]{}, \solh[second]{}} \\
		&\leq
		\dgnorm{\solh{}}^2.
	\end{align*}
	and with this~\eqref{eq:lts-Rpm-inv-bounds}. Using~\eqref{eq:Rpm-lower-bound},~\eqref{eq:lts-Rpm-properties-bb} and~\eqref{eq:lts-Rpm-inv-bounds} shows 
	\begin{align*}
		\cAmin(1 - \cflThetaLeapfrogCoarse^2) \dgnorm{\genop^n \genop[m][inv] \solh{}}^2 
		&\leq \dgip{\genop[m] \genop^n \genop[m][inv]\solh{}, \genop^n\genop[m][inv]\solh{}} \\
		&= \dgip{\genop[m][inv] \solh{}, \solh{}} \\
		&\leq \dgnorm{\solh{}}^2
	\end{align*}
    and hence \eqref{eq:lts-Rpm-bounds-b}.
\end{proof}
With the previous Lemma~\ref{lts:lem:Rpm-properties} we conclude the following stability estimate.
\begin{lemma}[Stability] \label{lts:lem:stability}
	Let \(\solh{n+1}\) be defined in~\eqref{eq:lts-scheme-transf}. Under the same assumptions as in Lemma~$\ref{lts:lem:Rpm-properties}$ the numerical solution is bounded by
	\begin{equation} \label{lts:eq:stability-a}
		\dgnorm{\solh{n+1}} \leq 
		\LTSStabConst
		\Bigl(	
			\norm{\solh{0}}[\RA]
			+
			\tau \sum_{j=0}^{n}  \dgnorm{\grhshn{j+1/2}}
		\Bigr).
	\end{equation}
\end{lemma}
\begin{proof}
	Since $\genop[m][inv]$ is invertible, the scheme~\eqref{eq:lts-scheme-transf} is equivalent to
	\begin{equation*}
		\solh{n+1} = \genop \solh{n}
		+
		\tau \genop[m][inv]\grhshn{n+1/2}, \qquad \genop = \genop[m][inv]\genop[p] .
	\end{equation*}
	Then, the discrete variation-of-constants formula yields
	\begin{equation}
		\solh{n+1} = \genop^{n+1} \solh{0}
		+
		\tau \sum_{j=0}^{n} \genop^{n-j}\genop[m][inv] \grhshn{j+1/2}.
	\end{equation}
	The claim now follows from \eqref{eq:lts-Rpm-bounds-a} and \eqref{eq:lts-Rpm-bounds-b} and the triangle inequality.
\end{proof}

\subsection{Error analysis of the local time-integration scheme}
In the following we denote the exact solution of~\eqref{eq:two-field-system} evaluated at time \(\tn{n}\) by
\begin{equation}\label{eq:exact-sol-short}
	\exsol{n} = \sol(\tn{n}), \qquad \sol = \pmat{\sol[first] & \sol[second]}.
\end{equation}
The error of the full discretization is given by
\begin{equation}\label{eq:lts-full-error}
	\err{n} = \exsol{n} - \solh{n} = \projerr{n} + \discerr{n},
	\qquad 
	\projerr{n} = \exsol{n} - \projLtwo\exsol{n}, 
	\qquad 
	\discerr{n} = \projLtwo\exsol{n} - \solh{n},
\end{equation}
where $\projerr{n}$ denotes the $\Ltwo$-projection error and \(\discerr{n}\) the discretization error respectively.

To define defects $\defect{n+1}$ we insert the $L^2$-projected exact solution into the numerical scheme \eqref{eq:lts-scheme-transf}. This yields
\begin{align}\label{eq:lts-exact-proj-scheme}
	\genop[m] \projLtwo \exsol{n+1} 
	= 
	\genop[p] \projLtwo \exsol{n}
	+ 
	\tau\grhshn{n+1/2}
	+ 
	\tau\defect{n+1} .
\end{align}
Subtracting~\eqref{eq:lts-scheme-transf} from~\eqref{eq:lts-exact-proj-scheme} yields the error recursion
\begin{equation} \label{eq:lts_error_recursion}
	\genop[m]\discerr{n+1} 
	= 
	\genop[p]\discerr{n}
	+
	\tau \defect{n+1}.
\end{equation}
This recursion is of the same form as \eqref{eq:lts-scheme-transf}. Hence, we can apply \Cref{lts:lem:stability} to bound the error. Unfortunately, it will turn out later, that we need a more careful inspection of the defects to deal with the cutoff functions within the \Friedrichs operators. Otherwise, this would lead to suboptimal error bounds.
\begin{theorem} \label{thm:lts-abstract_error_linear_full}
Let $\tau \leq \min\{\tsCFLfilfun, \tsCFLLFcoarse\}$ defined in~\eqref{eq:lfc-cfl} for some \(\cflThetaLeapfrogCoarse \in (0,\, 1)\). If one can decompose the defect into 
\begin{subequations} \label{eq:lts-abstract_error_linear_full}
	\begin{equation} \label{eq:lts-defectgen-split}
		\defect{j} = \defectsmooth{j} + \friedoph \defectfried{j} 
		\quad\text{with}\quad 
		\defectfried{j} = 
		\begin{pmatrix} 
			0 \\ 
			\defectfried[second]{j} 
		\end{pmatrix},
	\end{equation}
	then it holds
		\begin{align}\label{eq:lts-abstract_error_linear_full_v3}
			\dgnorm{\discerr{n+1}}
			&\leq \LTSStabConst
			\Bigl( 	
				\norm{\discerr{0}}[\RA]
				+
				\tau \sum_{j=1}^{n+1} \dgnorm{\defectsmooth{j}}
				+ 
				\dgnorm[second]{\defectfried[second]{1}}
				+
				\tau \sum_{j=2}^{n+1}  \dgnorm[second]{\dtau\defectfried[second]{j}}
			\Bigr) \nonumber \\
			&\qquad + \dgnorm[second]{ \defectfried[second]{n+1}}\ .
		\end{align}
	Here, 
	\begin{equation} \label{eq:discrete-time-derv-dt}
		\dtau\exsol[first]{n+1} 
		= \tfrac{1}{\tau}(\exsol[first]{n+1} - \exsol[first]{n}) 
	\end{equation}
	denotes the discrete time derivative of a function $\exsol[first]{}$.
\end{subequations}
\end{theorem}
\begin{proof}
	Solving the error  recursion \eqref{eq:lts_error_recursion} with the discrete variation-of-constants formula yields
	\begin{equation}  \label{eq:discerr-dvoc}
		\begin{aligned}
		\discerr{n+1} 
		& = \genop^{n+1} \discerr{0}
		+
		\tau \sum_{j=0}^{n} \genop^{n-j}\genop[m][inv] \defect{j+1} \\
		&= \genop^{n+1} \discerr{0}
		+
		\tau \sum_{j=0}^{n} \genop^{n-j}\genop[m][inv] \defectsmooth{j+1}
		+
		\sum_{j=0}^{n} \genop^{n-j} (\genop - \Id) \defectfried{j+1},
		\end{aligned}
	\end{equation}
	since 
	\begin{equation*}
		\tau \genop[m][inv] \friedoph \defectfried{j+1} 
		= \genop[m][inv](\genop[p] - \genop[m]) \defectfried{j+1}
		= (\genop - \Id) \defectfried{j+1}.
	\end{equation*}
	Now we use the well-known summation by parts formula: for suitable sequences $\{\rho_j\}_j$ and $\{\delta_j\}_j$ we have
	\begin{equation}\label{lts:eq:sum_by_parts}
		\sum_{j=0}^{n} \rho_{n-j} \delta_{j+1} 
		=
		r_{n} \delta_1  
		+
		\sum_{j=1}^{n} r_{n-j} ( \delta_{j+1} - \delta_j) ,
		\qquad r_k = \sum_{j=0}^k \rho_j.
	\end{equation}
	For $\rho_k=\genop^k(\genop-\Id)$, this yields $r_k = \genop^{k+1}- \Id$ and 
	\begin{equation*}
		\begin{aligned}
		\sum_{j=0}^{n} \genop^{n-j} (\genop - \Id) \defectfried{j+1} %
		&= 
		\genop^{n+1} \defectfried{1} - \defectfried{n+1}
			+ \tau \sum_{j=1}^{n} \genop^{n-j+1} \dtau \defectfried{j+1}.
		\end{aligned}
	\end{equation*}
	The bounds \eqref{eq:lts-Rpm-bounds-a} and \eqref{eq:lts-Rpm-bounds-b} imply \eqref{eq:lts-abstract_error_linear_full_v3}, since the first component of the defect \(\defectfried{}\) vanishes.
\end{proof}
For the discrete time-derivative \eqref{eq:discrete-time-derv-dt} of a sufficiently smooth function  \(\sol[first]\), a simple calculation shows the representations
\begin{equation}
\label{eq:discrete-to-contin-time-derv-all}  
		\dtau\exsol[first]{n+1} 
		= \int_0^1\derv{t}\sol[first](\tn{n} + \tau s)\d{s}
                , \qquad
                		\dtau^2\exsol[first]{n+1} 
		=\int_{-1}^{1}(1 - \abs{s})\derv[2]{t}\sol[first](\tn{n} + \tau s)\d{s}.
	\end{equation}
Note that the avarage $\derv{t}\exsol[first]{n+1/2}$ corresponds to the trapezoidal rule applied to the first integral in \eqref{eq:discrete-to-contin-time-derv-all}. It is well-known that the error 
\begin{subequations}\label{lts:eq:trapez-rule}
	\begin{equation}
		\defectTrapez[first]{n+1} =	\dtau\exsol[first]{n+1} 
		- 
		\derv{t}\exsol[first]{n+1/2} 
	\end{equation}
	satisfies
	\begin{equation}\label{lts:eq:trapez-defect-bound}
		\dgnorm[first]{\defectTrapez[first]{n+1}} 
		\leq 
		\frac{\tau^2}{8}\int_{0}^{1}\dgnorm[first]{\derv[3]{t}\sol[first](\tn{n} + \tau s)}\d{s} \ .
	\end{equation}
\end{subequations}
Moreover, we define 
\begin{subequations} \label{eq:auxfilfuncardi-all}
	\begin{equation} \label{eq:auxfilfuncardi}
		\auxfilfuncardiop = \auxfilfuncardiop*[m], 
		\qquad
		\text{with}
		\qquad
		\auxfilfuncardi(z) = \frac{\auxfilfun(z)-1}{z}, \qquad z>0,
	\end{equation}
	and set $ \auxfilfuncardi(0)=\auxfilfun'(0)$. In addition, we will use
	\begin{equation}
		\widetilde{\auxfilfuncardiop } = \auxfilfuncardi(-\tau^2\cutoffModified\friedoph[first]\friedoph[second]\cutoffModified).
	\end{equation}
\end{subequations}
Note that the nonzero eigenvalues of
\(\cutoffModified\friedoph[first]\friedoph[second]\cutoffModified\) and \(\friedoph[second]\cutoffModified\friedoph[first]\) coincide.

Furthermore, we recall the consistency and approximation properties of the \dG discretized \Friedrichs operators \(\friedoph[dummy]\),
with \(\sol[dummy] \in \{\sol[first], \sol[second]\}\),
\begin{subequations}\label{eq:dg-op-props}
	\begin{align}
		\friedoph[dummy]\sol[dummy] 
		&= 
		\projLtwo\friedop[dummy]\sol[dummy]
		&& 
		\text{for all }\sol[dummy]\in \frieddom[dummy]\cap\brokensobolev[dummy]{1}, \label{eq:consistency} \\
		\norm{\friedoph[dummy]\projerr[dummy]{}}[\domain]
		&\leq 		
		\Capprox[dummy]\seminormHkTh{\dGdegree+1}{\meshdiam[\dGdegree]\sol[dummy]}
		&& 
		\text{for all }\sol[dummy]\in\frieddom[dummy]\cap\brokensobolev[dummy]{\dGdegree+1},\label{eq:approx_prop}
	\end{align}
\end{subequations}
cf.~\cite{DiPE12, MFO2023}.

To derive a representation of the defect \(\defect{j}\) defined in~\eqref{eq:lts-exact-proj-scheme}, we write \eqref{eq:genop-splitting-all} as a perturbation of the \CN scheme.
\begin{lemma}\label{Lem:CN-split}
	The operators \(\genop[pm]\) defined in \eqref{eq:genop-splitting-all} satisfy
	\begin{subequations}\label{eq:perturbation-of-CN}
		\begin{equation}\label{eq:perturbation-of-CN-a}
			\genop[pm] = \genop[pm][CN] + \friedoph\perturbationop{LTI},
		\end{equation}
		where we have 
		\begin{equation}\label{eq:perturbation-of-CN-b}
			\genop[pm][CN] = \Id \pm \tauh\friedoph 
			\quad\text{and}\quad 
			\perturbationop{LTI} = \ts^2
			\begin{pmatrix}
				0 & 0 \\
				0 & \frac14\cutoffLeapfrog\Id[second] - \widetilde{\auxfilfuncardiop}\cutoffModified
			\end{pmatrix}
			\friedoph .
		\end{equation}
	\end{subequations}
\end{lemma}
\begin{proof}
	The resprentation follows from the definiton of \(\auxfilfuncardi\) in \eqref{eq:auxfilfuncardi} and \(\cutoffModified = \cutoffModified^2\) since
	\begin{align*}
		\auxfilfunop[m]
		=
		\Id[first] -\ts^2\friedoph[second]\cutoffModified\friedoph[first]\auxfilfuncardiop[m]
		= 
		\Id[first] -\ts^2\friedoph[second]\cutoffModified\widetilde{\auxfilfuncardiop}\cutoffModified\friedoph[first],
	\end{align*}
	where the second equality follows from~\cite[Cor. 1.34]{Hig08}.
\end{proof}
With \Cref{Lem:CN-split}, we can split the defect \(\defect{j}\) defined in~\eqref{eq:lts-exact-proj-scheme} as 
\begin{subequations}\label{eq:defect-split-CN-LTI}
	\begin{equation}\label{eq:defect-split-CN-LTI-a}
		\defect{j} 
		=
		\defect[CN]{j} + \friedoph\defect[LTI]{j}
	\end{equation}
	with
	\begin{align}
		\defect[CN]{j}
		&= 
		\projLtwo\dtau\exsol{j} - \friedoph\projLtwo\exsol{j-1/2} - \grhshn{j-1/2}, \label{eq:defect-split-CN} \\
		\defect[LTI]{j}
		&= 
		\perturbationop{LTI}\projLtwo\dtau\exsol{j} \nonumber \\
		&=
		\perturbationop{LTI}\dtau\exsol{j} - \perturbationop{LTI}\dtau\projerr{j} . \label{eq:defect-split-LTI}
	\end{align}
\end{subequations}
The first component of \(\defect[LTI]{j}\) vanishes such that we can apply~\Cref{thm:lts-abstract_error_linear_full} later. 

We now have to bound the defect~\eqref{eq:defect-split-CN-LTI-a}. Following~\cite[Lemma 12.2]{MFO2023} the defect~\eqref{eq:defect-split-CN} stemming from the \CN scheme can be bounded under appropriate regularity assumptions on the exact solution of \eqref{eq:two-field-system} as 
\begin{equation}\label{eq:CN-defect-estimate}
	\dgnorm{\defectCN{j}}
	\leq
	\Capprox\seminormHkTh{\dGdegree+1}{\meshdiam[\dGdegree]\exsol[first]{j-1/2}} 
	+
	\tfrac{\tau^2}{8}\int_{0}^{1}\dgnorm[second]{\derv[3]{t}\sol[second](\tn{j-1} + \tau s)}\d{s}.
\end{equation}
To bound the defect~\eqref{eq:defect-split-LTI} we need a bound on \(\auxfilfuncardi\) introduced in~\eqref{eq:auxfilfuncardi}.
\begin{definition} \label{def:cphi}
	With $\hbeta$ from \Cref{def:ctheta}, we define $\cAuxfilfuncardi$ as the smallest constant such that for $\auxfilfuncardi $ defined in \eqref{eq:auxfilfuncardi} it holds
	\begin{equation}\label{eq:auxfilfuncardi-bound}
		\abs{\auxfilfuncardi(z)} \le \cAuxfilfuncardi
		\qquad \text{for all }
		z \in [0,\hbeta^2] \cap \real.
	\end{equation}
\end{definition}
Note that such a bound exists since \(\auxfilfuncardi\) is continuous on \([0,\hbeta^2]\) and hence bounded.

Now we are able to state a bound for the defect~\eqref{eq:defect-split-LTI} under appropriate regularity assumptions on the exact solution of \eqref{eq:two-field-system}.
\begin{lemma}\label{Lem:lts-defect-bounds}
	Let \(\sol = \begin{pmatrix} \sol[first] & \sol[second] \end{pmatrix}\) with
	\begin{equation*}
		\sol \in C(\tn{0}, T; D(\friedop) \cap \brokensobolev{\dGdegree+1}) 
				\cap C^3(\tn{0}, T; \Ltwovec{\domain})
	\end{equation*}
	be the solution of~\eqref{eq:two-field-system} and \(\tau \leq \tsCFLfilfun\) defined in \eqref{def:lts-CFL}. We further assume
	\begin{equation*}
		\grhs[second]\in C^2(\tn{0}, T; \Ltwovec[second]{\domain}) .
	\end{equation*}
	Then the defect~\eqref{eq:defect-split-LTI} is bounded by
	\begin{subequations}
	\begin{align}
		\dgnorm[second]{\defect[LTI]{n+1}} 
		&\leq 
		\tau^2 \cAuxfilfuncardimax 
		\Bigl(
			\Capprox\seminormHkTh{\dGdegree+1}{\meshdiam[\dGdegree]\dtau\exsol[first]{n+1}} 
			+ 
			\int_0^1\dgnorm[second]{\derv[2]{t}\sol[second](\tn{n} + \tau s)}\d{s} \label{lts:eq:friedDefect-bound}  \\
			&\phantom{
				\leq\tau^2 \cAuxfilfuncardimax 
				\Big(
				\Capprox\seminormHkTh{\dGdegree+1}{\meshdiam[k]\dtau\exsol[first]{n+1}} 
			}
			+ 
			\int_0^1\dgnorm[second]{\derv{t}\grhs[second](\tn{n} + \tau s)}\d{s}
		\Bigr) \nonumber \\
		\dgnorm[second]{\dtau\defect[LTI]{n+1}} 
		&\leq 
		\tau^2 \cAuxfilfuncardimax 
		\Bigl(
			\Capprox\seminormHkTh{\dGdegree+1}{\meshdiam[k]\dtau^2\exsol[first]{n+1}} 
			+ 
			\int_{-1}^{1} \dgnorm[second]{\derv[3]{t}\sol[second](\tn{n} + \tau s)}\d{s} \label{lts:eq:dtau-friedDefect-bound} \\
			&\phantom{
				\leq\tau^2 \cAuxfilfuncardimax 
				\Big(\Capprox\seminormHkTh{\dGdegree+1}{\meshdiam[\dGdegree]\dtau^2\exsol[first]{n+1}}
			}
			+ 
			\int_{-1}^{1} \dgnorm[second]{\derv[2]{t}\grhs[second](\tn{n} + \tau s)}\d{s}
		\Bigr) \nonumber
	\end{align}
	\end{subequations}
	with $\cAuxfilfuncardimax = \sqrt{2}\max\{\tfrac14, \cAuxfilfuncardi\}$.
\end{lemma}
\begin{proof}
	With~\eqref{eq:consistency},~\eqref{eq:approx_prop},~\eqref{eq:auxfilfuncardi-bound}, and~\eqref{eq:two-field-system} we get
	\begin{align*}
		\dgnorm[second]{\defect[LTI]{n+1}} 
		&\leq
		\tausf\dgnorm[second]{\cutoffLeapfrog\projLtwo\dtau\derv[]{t}\exsol[second]{n+1}} 
		+
		\tau^2\cAuxfilfuncardi\dgnorm[second]{\cutoffModified\projLtwo\dtau\derv[]{t}\exsol[second]{n+1}} \\
		&\quad
		+
		\tausf\dgnorm[second]{\cutoffLeapfrog\projLtwo\dtau\grhs[second]^{n+1}}
		+
		\tau^2\cAuxfilfuncardi\dgnorm[second]{\cutoffModified\projLtwo\dtau\grhs[second]^{n+1}} \\
		&\quad+
		\tau^2\cAuxfilfuncardi\Capprox\seminormHkTh{\dGdegree+1}{\cutoffModified\meshdiam[\dGdegree]\dtau\exsol[first]{n+1}} 
		+
		\tausf\Capprox\seminormHkTh{\dGdegree+1}{\cutoffLeapfrog\meshdiam[\dGdegree]\dtau\exsol[first]{n+1}}
	\end{align*}
	and in the same manner
	\begin{align*}
		\dgnorm[second]{\dtau\defect[LTI]{n+1}} 
		&\leq
		\tausf\dgnorm[second]{\cutoffLeapfrog\projLtwo\dtau^2\derv[]{t}\exsol[second]{n+1}} 
		+
		\tau^2\cAuxfilfuncardi\dgnorm[second]{\cutoffModified\projLtwo\dtau^2\derv[]{t}\exsol[second]{n+1}} \\
		&\quad
		+
		\tausf\dgnorm[second]{\cutoffLeapfrog\projLtwo\dtau^2\grhs[second]^{n+1}}
		+
		\tau^2\cAuxfilfuncardi\dgnorm[second]{\cutoffModified\projLtwo\dtau^2\grhs[second]^{n+1}} \\
		&\quad +
		\tau^2\cAuxfilfuncardi\Capprox\seminormHkTh{\dGdegree+1}{\cutoffModified\meshdiam[\dGdegree]\dtau^2\exsol[first]{n+1}}
		+
		\tausf\Capprox\seminormHkTh{\dGdegree+1}{\cutoffLeapfrog\meshdiam[\dGdegree]\dtau^2\exsol[first]{n+1}} \ .
	\end{align*}
	Using~\eqref{eq:discrete-to-contin-time-derv-all} %
        completes the proof.
\end{proof}
Now we state our main theorem which yields under certain regularity assumptions convergence of order two in time and order \(\dGdegree\) in space if we choose as \dG polynomial degree \(\dGdegree\).
\begin{theorem}\label{Thm:error-estimate}
	By the same regularity assumptions as in Lemma~$\ref{Lem:lts-defect-bounds}$ together with \(\ts \leq \min\{\tsCFLLFcoarse, \tsCFLfilfun\}\) and \(\solh{0} = \projLtwo\sol(\tn{0})\), the error of the full discretization satisfies 
	\begin{align}
		\dgnorm{\sol(\tn{n+1}) - \solh{n+1}} 
		&\leq C(\tau^2 + \maxmeshdiam[\dGdegree])
	\end{align}
	with a constant \(C\) independent of \(\tau\) and \(\meshdiam\).
\end{theorem}
\begin{proof}
	The projection error in \eqref{eq:lts-full-error} is bounded by
	\begin{equation*}
		\dgnorm{\projerr{n+1}} \leq \Cproj\seminormHkTh{\dGdegree+1}{\meshdiam[k+1]\exsol{n+1}},
	\end{equation*}	
	\cf~\cite[Lemma 3.2]{MFO2023}.
	With Theorem~\ref{thm:lts-abstract_error_linear_full}, Lemma~\ref{Lem:lts-defect-bounds}, and \(\solh{0} = \projLtwo\exsol{0}\) we get for the discretization error
	\begin{align*}
		\dgnorm{\discerr{n+1}}
		&\leq \LTSStabConst
		\biggl( 	
			\tau \sum_{j=1}^{n+1} \dgnorm{\defectsmooth{j}}
			+ 
			\dgnorm[second]{\defectfried[second]{1}}
			+
			\tau \sum_{j=2}^{n+1} \dgnorm[second]{ \dtau \defectfried[second]{j}}
		\biggr)
		+ 
		\dgnorm[second]{\defectfried[second]{n+1}} \\
		&\leq \LTSStabConst
		\biggl(
			\frac{\tau^2}{8}\normLoneLtwo{\derv[3]{t}\sol} \\
			&\qquad\qquad+ 
			2\tau^2\cAuxfilfuncardimax\Big(
				\normLoneLtwo{\derv[3]{t}\sol[second]}
				+
				\normLoneLtwo{\derv[2]{t}\grhs[second]}
			\Big) \\
			&\qquad\qquad+
			\cAuxfilfuncardimax\Capprox\max_{j=0,\dots,n}\seminormHkTh{\dGdegree+1}{\meshdiam[\dGdegree]\tau^2\dtau\exsol[first]{j+1}} \\
			&\qquad\qquad+
			\cAuxfilfuncardimax\tau^2\Big(
				\max_{s\in[\tn{0}, T]}\dgnorm[second]{\derv[2]{t}\sol[second](s)} 
				+
				\max_{s\in[\tn{0}, T]}\dgnorm[second]{\derv[1]{t}\grhs[second](s)}
			\Big) \\
			&\qquad\qquad+
			\Capprox\tau\sum_{j=0}^{n}\seminormHkTh{\dGdegree+1}{\meshdiam[\dGdegree]\exsol{j+1/2}}
			+
			\cAuxfilfuncardimax\Capprox\tau\sum_{j=1}^{n}\seminormHkTh{\dGdegree+1}{\meshdiam[\dGdegree]\tau^2\dtau^2\exsol[first]{j+1}}
		\biggr).
	\end{align*}
	This proves the statement.
\end{proof}

\begin{remark}
	If we use the midpoint evaluation \(\grhshn{n+1/2} = \grhshn{}(\tn{n+1/2})\) in~\eqref{eq:lts-scheme} instead of the average~\eqref{eq:avg-rhs-def}, we get an additional term \(\tau^2\normLoneLtwo{\derv[2]{t}\grhs}\) in the bound of Theorem~\ref{Thm:error-estimate}. Nevertheless, this modification leads to the same order of convergence.
\end{remark}

\subsection{Stability and convergence of local time-integration schemes}\label{subsec:special-filfuns}

So far, we proved our theoretical results for general filter functions \(\filfun\). In this section, we investigate the constants from
\Cref{def:ctheta}  for the special case of using \lfc polynomials and for a rational function from our previous work \cite{CarH22,CarH24}.

We start with considering the locally-implicit method from~\cite{HocS16}, where, for all  $z \in [0,\infty)$, it holds
\begin{equation}\label{eq:CN:filfun}
    \filfun(z) = (1 + \tfrac{z}{4})^{-1} > 0, \qquad 
    \auxfilfun(z) = 1, \qquad  \auxfilfuncardi(z) = 0.
\end{equation}
Thus we obtain \(\hbeta = \infty\) with the constants \(\cAuxfilfunmin = 1\) and \(\cAuxfilfuncardi = 0\). Inserting these constants into \Cref{def:lts-CFL} we get exactly the CFL conditions from the literature, i.e., \cite[Assumption~11.26]{MFO2023} for \Friedrichs systems or from \cite{HocS16} for the special case of \Maxwells equations. This indicates, that our new general theory does not require stronger assumptions than for the known results.

Now we turn to the new \lts methods. Here, motivated by \cite{CarH22}, we choose the filter function as the following polynomial of degree $p$
\begin{equation}\label{eq:lfc-polynomials}
	\filfun(z)z 
	= 
	\Pp(z)z 
	= 
	2 - \frac{2}{\Tp(\nup)}\Tp\Bigl(\nup - \frac{z}{\alphap}\Bigr),
	\qquad
	\alphap 
	= 
	2\frac{\Tp^{\prime}(\nup)}{\Tp(\nup)},
\end{equation}
with a stabilization parameter \(\nup > 1\). With \(\Tp\), we denote the $p$th Chebychev polynomial of first kind. 
\begin{theorem} \label{thm:results-lts}
	\begin{subequations}
	Let $\filfun$ be given by \eqref{eq:lfc-polynomials} for some $p\in\bbN$ and \(\nup > 1\). If we choose 
	\begin{equation}\label{eq:lfc-filfunBound-constants}
		\cAuxfilfunmin = \tfrac12\bigl(1 - \tfrac{1}{\Tp(\nup)}\bigr) \in (0, \tfrac12),
	\end{equation}
	then, \eqref{eq:ctheta} and \eqref{eq:auxfilfuncardi-bound} are satisfied for 
	\begin{equation}\label{eq:lfc-beta}
		\hbeta^2 = \alphap(\nup + 1) 
		\qquad \text{and} \qquad
		\cAuxfilfuncardi = \tfrac14\left(\tfrac{1}{\cAuxfilfunmin} - 1\right). 
	\end{equation}
	\end{subequations}
\end{theorem}
\begin{proof}
	Following~\cite[Lemmas 5.1, 5.4]{CarH22} we have
	\begin{equation*} 
		\filfun(z)z \leq 4(1 - \cAuxfilfunmin), \qquad 0 < \filfun(z) \leq 1,
		\qquad 
		0 \leq z \leq \hbeta^2
	\end{equation*}
	and hence
	\begin{equation*}
		\auxfilfun(z) 
		= 
		\filfun(z)^{-1} \bigl(1 - \tfrac{z}{4} \filfun(z) \bigr) 
		\geq 
		1 - \tfrac{z}{4} \filfun(z) 
		\geq 
		\cAuxfilfunmin \ .
	\end{equation*}
	The formula for $\cAuxfilfuncardi$ can be shown analogously to the proof of~\cite[Lemma 5.4]{CarH22}.
\end{proof}
Note that for $p=1$ we have $\filfun \equiv 1$ and the scheme~\eqref{eq:lts-scheme} is just the \leapfrog method (which is independent of the stabilization parameter \(\nup\)). If we choose \(\cAuxfilfunmin = 1 - \cflThetaLeapfrogAll^2\), for \(\cflThetaLeapfrogAll^2\in(0, 1)\), we get \(\hbeta[\filfun]^2 = 4(1 - \cAuxfilfunmin) = 4\cflThetaLeapfrogAll^2\).

We collect all relevant details on the different LTI schemes and their constants in \Cref{tab:methods-filfuns-vals}.

If, for \(p\geq 2\), we choose 
\begin{equation}\label{eq:special-stab-param}
	\nup = 1 + \frac{\lfcstabparam^2}{2\lfcdegree^2}, \qquad \lfcstabparam > 0,
\end{equation}
then $\cAuxfilfunmin$ defined in \eqref{eq:lfc-filfunBound-constants} can be bounded independently of \(\lfcdegree\), see~\cite[Lemma 5.5]{CarH22} for a more in-depth view. In this case it holds 
\begin{equation*}
	\hbeta[\filfun]^2 = \hbeta[p]^2 \geq 2\lfcdegree \geq 4.
\end{equation*}
Thus, the assumption of Lemma~\ref{Lemma:Rpm-bound-lf-all} on $\filfun$ is satisfied for the \lfc polynomials \eqref{eq:lfc-polynomials}.
\begin{remark}[Implementation]
	\begin{enumerate}
          \item It is important to note that one does not have to
          evaluate the polynomials~\eqref{eq:lfc-polynomials} of
          degree \(\lfcdegree - 1\) in each step.  Instead one
          calculates the action of a vector \(\solvector[first]{}\) on
          \(\filfun(-\tau^2\scdordermatrixmodified)\) in an efficient
          way by a three-term recurrence relation,
          see~\cite[Algo. 4.2]{Car21} for details.
		Here \(\scdordermatrixmodified\) denotes the system matrix of \(\friedopsecondorder[m]\) one obtains after choosing an appropriate basis of the \dG space \(\dgspace\), see~\cite{HocS19}.
		\item By sorting the degrees-of-freedom in an appropriate way~\cite{HocS19, Stu17} one can see that \(\filfun(-\tau^2\scdordermatrixmodified)\) only acts on the few fine elements in \(\meshFine\) plus two additional layers.
	\end{enumerate}
\end{remark}

\section{Numerical examples}\label{sec:numerical-examples}
At last, we verify our findings numerically with three examples. First we  substantiate the error bounds of \Cref{Thm:error-estimate} and we study the influence of stabilization of the \lfc \lts (LFC-LTS) method. Afterwards, we investigate the efficiency of the LFC-LTS method compared to the locally implicit (LI) and the original \leapfrog scheme. The linear systems in the LI method are solved with the conjugate gradient method without preconditioning since it required only a few iterations (not more than four in our examples). Note that it is essential to run the conjugate gradient method with the correct inner-product induced by the mass matrix.

The codes to reproduce our results are available at
\begin{center}
    \url{https://gitlab.kit.edu/malik.scheifinger/dg-lts-maxwell}
\end{center}
The software is based on the FEM library \href{https://www.dealii.org}{\texttt{deal.II}}~\cite{dealII23} at version 9.5 and the Maxwell toolbox \href{https://gitlab.kit.edu/kit/ianm/ag-numerik/projects/dg-maxwell/timaxdg}{\texttt{TiMaxdG}}~\cite{TiMaxdG}. Since we use \texttt{deal.II}, all our examples are done with rectangular mesh elements.

\subsection{Linear Maxwells equations}

Linear \Maxwells equations in transverse-electric (TE) mode, see~\cite[\S 2.3]{Nie09}, are given by
\begin{subequations}\label{eq:TE-system}
    \begin{align}
        \permitt\derv{t} \exE_x &= \partial_y \exH_z - \exJ_x, && \domain\times(0,T), \\
        \permitt\derv{t} \exE_y &= -\partial_x \exH_z - \exJ_y, && \domain\times(0,T), \\
        \permeab\derv{t} \exH_z &= \partial_y \exE_x - \partial_x \exE_y, && \domain\times(0,T), \\
        \exE(0) &= \exE^0, \quad \exH(0) = \exH^0, && \domain,
        \\
        \exE \times \normal{} &= 0, && \partial\domain\times(0,T) .
    \end{align}
\end{subequations}
As computational spatial domain we choose \(\domain = (0, 1)^2\) and final time \(T = 1\). We set \(\permitt = \permeab = 1\) and use initial values 
\begin{subequations}\label{eq:TE-data-IV}
    \begin{align}
        \exE_x^0(x, y, t) &= \cos(2\pi x)\sin(2\pi y), \\
        \exE_y^0(x, y, t) &= -\sin(2\pi x)\cos(2\pi y), \\
        \exH_z^0(x, y, t) &= 4\pi\cos(2\pi x)\cos(2\pi y),
    \end{align}
\end{subequations}
as well as right-hand sides
\begin{subequations}\label{eq:TE-data-RHS}
    \begin{align}
        \exJ_x(x, y, t) &= -(1 + 8\pi^2)\cos(2\pi x)\sin(2\pi y)\mathrm{e}^t, \\
        \exJ_y(x, y, t) &= (1 + 8\pi^2)\sin(2\pi x)\cos(2\pi y)\mathrm{e}^t .
    \end{align}
\end{subequations}
The exact solution to~\eqref{eq:TE-system} with~\eqref{eq:TE-data-IV} and~\eqref{eq:TE-data-RHS} is given by a variant of the cavity solution~\cite{HocS16}
\begin{subequations}\label{eq:TE-exact-sol}
    \begin{align}
        \exE_x(x, y, t) &= \cos(2\pi x)\sin(2\pi y)\mathrm{e}^t, \\
        \exE_y(x, y, t) &= -\sin(2\pi x)\cos(2\pi y)\mathrm{e}^t, \\
        \exH_z(x, y, t) &= 4\pi\cos(2\pi x)\cos(2\pi y)\mathrm{e}^t .
    \end{align}
\end{subequations}
Thereby we can compute the exact \(L^2\)-error of our scheme.

\begin{figure}[!htb]
    \centering
    \includegraphics[width=6cm, height=6cm]{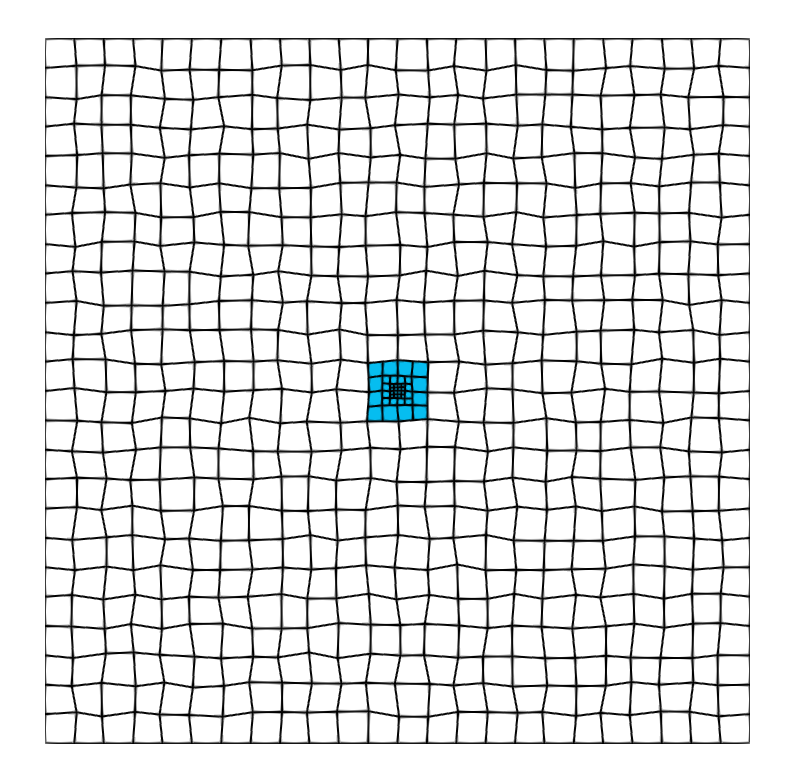}
    \raisebox{0.5cm}{{\includegraphics[width=5cm, height=5cm]{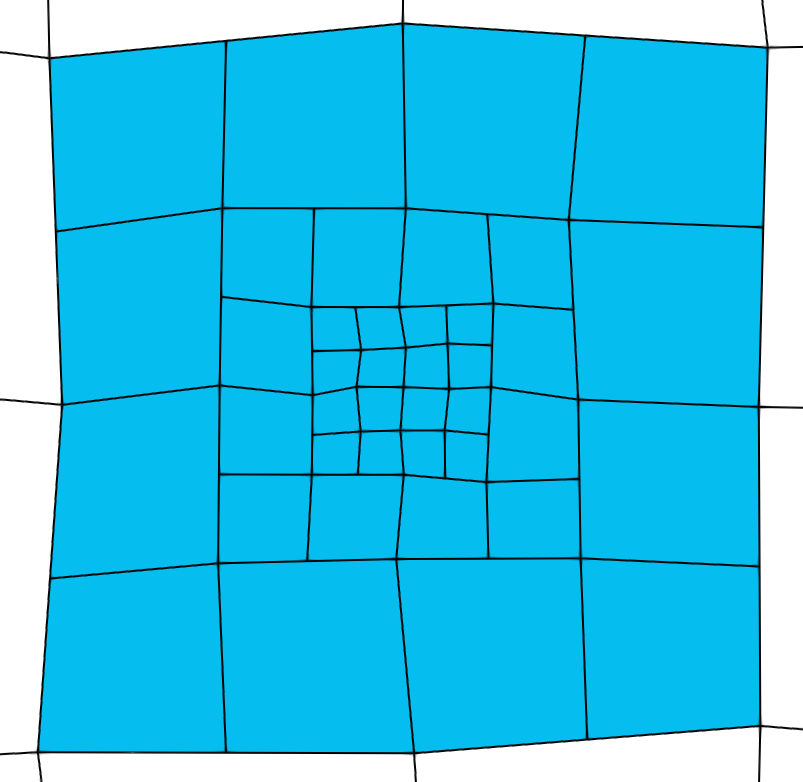}}}
    \caption{Locally refined mesh of the domain \((0,1)^2\).}
    \label{fig:2D-mesh}
\end{figure}

We use the spatial mesh illustrated in \Cref{fig:2D-mesh} in which we apply threefold refinement in the blue central box. The mesh has been randomly perturbed to counteract super convergence effects of the space discretization. Moreover, the \dG degree is choosen as \(\dGdegree = 5\). 

\begin{figure}[!htb]
    \centering
    \begin{tikzpicture}

\definecolor{crimson2143940}{RGB}{214,39,40}
\definecolor{darkgray176}{RGB}{176,176,176}
\definecolor{darkorange25512714}{RGB}{255,127,14}
\definecolor{forestgreen4416044}{RGB}{44,160,44}
\definecolor{gray}{RGB}{128,128,128}
\definecolor{lightgray204}{RGB}{204,204,204}
\definecolor{mediumpurple148103189}{RGB}{148,103,189}
\definecolor{steelblue31119180}{RGB}{31,119,180}

\begin{axis}[
legend cell align={left},
legend style={
  at={(1.12,-0.2)},
  draw=lightgray204
},
legend columns=-1,
log basis x={10},
log basis y={10},
tick align=inside,
tick pos=left,
x grid style={darkgray176},
xlabel={time stepsize $\tau$},
xmin=8.31566529016915e-05, xmax=0.0028,
xmode=log,
xtick style={color=black},
xtick={0.0001,0.001},
xticklabels={
  \(\displaystyle {10^{-4}}\),
  \(\displaystyle {10^{-3}}\)
},
y grid style={darkgray176},
ylabel={max \(\displaystyle L^2\) error},
ymin=3e-07, ymax=0.00003,
ymode=log,
ytick style={color=black}
]
\path [draw=black, semithick, dash dot]
(axis cs:0.0020184163717259,1e-07)
--(axis cs:0.0020184163717259,0.0001);

\addplot [thick, steelblue31119180, mark=*, mark size=1, mark options={solid}]
table {%
0.00249269 1e+20
0.00226772 1e+20
0.002094 1e+20
0.00206305 1e+20
0.002018 1e+20
0.00187686 1e+20
0.001862 1e+20
0.00170747 1e+20
0.00155337 1e+20
0.00141318 1e+20
0.00128564 1e+20
0.00116961 1e+20
0.00106405 1e+20
0.00096802 1e+20
0.000933 1e+20
0.00088065 1e+20
0.00080117 1e+20
0.00072887 1e+20
0.00066309 1e+20
0.00060324 1e+20
0.0005488 1e+20
0.00049927 1e+20
0.000471 1e+20
0.00045421 1e+20
0.00041322 1e+20
0.00037592 1e+20
0.000342 1e+20
0.00031113 1e+20
0.00028305 1e+20
0.000266 4.7e-07
0.0002575 4.7e-07
0.00023426 4.6e-07
0.00021312 4.5e-07
0.00019389 4.4e-07
0.00017639 4.4e-07
0.00016047 4.4e-07
0.00014599 4.3e-07
0.00013281 4.3e-07
0.00012083 4.3e-07
0.00010992 4.3e-07
0.0001 4.3e-07
};
\addlegendentry{\leapfrog}

\addplot [thick, darkorange25512714, mark=*, mark size=1, mark options={solid}]
table {%
0.00249269 1e+20
0.00226772 1e+20
0.002094 1e+20
0.00206305 1e+20
0.002018 1e+20
0.00187686 1e+20
0.001862 1e+20
0.00170747 1e+20
0.00155337 1e+20
0.00141318 1e+20
0.00128564 1e+20
0.00116961 1e+20
0.00106405 1e+20
0.00096802 1e+20
0.000933 1e+20
0.00088065 1e+20
0.00080117 1e+20
0.00072887 1e+20
0.00066309 1e+20
0.00060324 1e+20
0.0005488 1e+20
0.00049927 1e+20
0.000471 7.4e-07
0.00045421 7.1e-07
0.00041322 6.3e-07
0.00037592 5.8e-07
0.000342 5.3e-07
0.00031113 5e-07
0.00028305 4.8e-07
0.000266 4.7e-07
0.0002575 4.7e-07
0.00023426 4.6e-07
0.00021312 4.5e-07
0.00019389 4.4e-07
0.00017639 4.4e-07
0.00016047 4.4e-07
0.00014599 4.3e-07
0.00013281 4.3e-07
0.00012083 4.3e-07
0.00010992 4.3e-07
0.0001 4.3e-07
};
\addlegendentry{$p = 2$}

\addplot [thick, forestgreen4416044, mark=+, mark size=2, mark options={solid}]
table {%
0.00249269 1e+20
0.00226772 1e+20
0.002094 1e+20
0.00206305 1e+20
0.002018 1e+20
0.00187686 1e+20
0.001862 1e+20
0.00170747 1e+20
0.00155337 1e+20
0.00141318 1e+20
0.00128564 1e+20
0.00116961 1e+20
0.00106405 1e+20
0.00096802 1e+20
0.000933 2.38e-06
0.00088065 2.13e-06
0.00080117 1.78e-06
0.00072887 1.49e-06
0.00066309 1.26e-06
0.00060324 1.07e-06
0.0005488 9.2e-07
0.00049927 8e-07
0.000471 7.4e-07
0.00045421 7e-07
0.00041322 6.3e-07
0.00037592 5.7e-07
0.000342 5.3e-07
0.00031113 5e-07
0.00028305 4.8e-07
0.000266 4.7e-07
0.0002575 4.7e-07
0.00023426 4.6e-07
0.00021312 4.5e-07
0.00019389 4.4e-07
0.00017639 4.4e-07
0.00016047 4.4e-07
0.00014599 4.3e-07
0.00013281 4.3e-07
0.00012083 4.3e-07
0.00010992 4.3e-07
0.0001 4.3e-07
};
\addlegendentry{$p = 4$}

\addplot [thick, crimson2143940, mark=x, mark size=2, mark options={solid}]
table {%
0.00249269 1e+20
0.00226772 1e+20
0.002094 1e+20
0.00206305 1e+20
0.002018 1e+20
0.00187686 1e+20
0.001862 9.32e-06
0.00170747 7.83e-06
0.00155337 6.48e-06
0.00141318 5.37e-06
0.00128564 4.45e-06
0.00116961 3.69e-06
0.00106405 3.07e-06
0.00096802 2.55e-06
0.000933 2.37e-06
0.00088065 2.13e-06
0.00080117 1.78e-06
0.00072887 1.49e-06
0.00066309 1.26e-06
0.00060324 1.07e-06
0.0005488 9.2e-07
0.00049927 8e-07
0.000471 7.3e-07
0.00045421 7e-07
0.00041322 6.3e-07
0.00037592 5.7e-07
0.000342 5.3e-07
0.00031113 5e-07
0.00028305 4.8e-07
0.000266 4.7e-07
0.0002575 4.7e-07
0.00023426 4.5e-07
0.00021312 4.5e-07
0.00019389 4.4e-07
0.00017639 4.4e-07
0.00016047 4.4e-07
0.00014599 4.3e-07
0.00013281 4.3e-07
0.00012083 4.3e-07
0.00010992 4.3e-07
0.0001 4.3e-07
};
\addlegendentry{$p = 8$}

\addplot [thick, mediumpurple148103189, mark=star, mark size=2, mark options={solid}]
table {%
0.00249269 1e+20
0.00226772 1e+20
0.002094 1e+20
0.00206305 1e+20
0.002018 1.092e-05
0.00187686 9.44e-06
0.001862 9.32e-06
0.00170747 7.83e-06
0.00155337 6.48e-06
0.00141318 5.37e-06
0.00128564 4.45e-06
0.00116961 3.69e-06
0.00106405 3.07e-06
0.00096802 2.55e-06
0.000933 2.37e-06
0.00088065 2.12e-06
0.00080117 1.78e-06
0.00072887 1.49e-06
0.00066309 1.26e-06
0.00060324 1.07e-06
0.0005488 9.2e-07
0.00049927 8e-07
0.000471 7.3e-07
0.00045421 7e-07
0.00041322 6.3e-07
0.00037592 5.7e-07
0.000342 5.3e-07
0.00031113 5e-07
0.00028305 4.8e-07
0.000266 4.7e-07
0.0002575 4.7e-07
0.00023426 4.5e-07
0.00021312 4.5e-07
0.00019389 4.4e-07
0.00017639 4.4e-07
0.00016047 4.4e-07
0.00014599 4.3e-07
0.00013281 4.3e-07
0.00012083 4.3e-07
0.00010992 4.3e-07
0.0001 4.3e-07
};
\addlegendentry{$p = 9$}

\draw (0.0014, 4.8e-6) -- (0.0007, 1.2e-6) -- (0.0014, 1.2e-6) -- cycle;
\draw (0.0011, 2.6e-6) coordinate[label=below:$2$];

\end{axis}

\end{tikzpicture}
    \caption{Error of the numerical solution of~\eqref{eq:TE-system} with initial data given by~\eqref{eq:TE-exact-sol} obtained by the \leapfrog method (blue) and the \lts method~\eqref{eq:lts-scheme} with filter~\eqref{eq:lfc-polynomials}, polynomial degrees \(\lfcdegree = 2\) (orange), \(p= 4\) (green), \(p = 8\) (red), \(p = 9\) (purple), and stabilization~\eqref{eq:special-stab-param} with \(\lfcstabparam = 1\). The space discretization is done with \dG degree \(\dGdegree = 5\) and a three times at the center locally refined mesh, see~\Cref{fig:2D-mesh}. The dash-dotted line depicts the maximal stable time stepsize of \leapfrog method used on the coarse mesh \(\meshLeapfrog\).}
    \label{fig:lfc-convergence}
\end{figure}
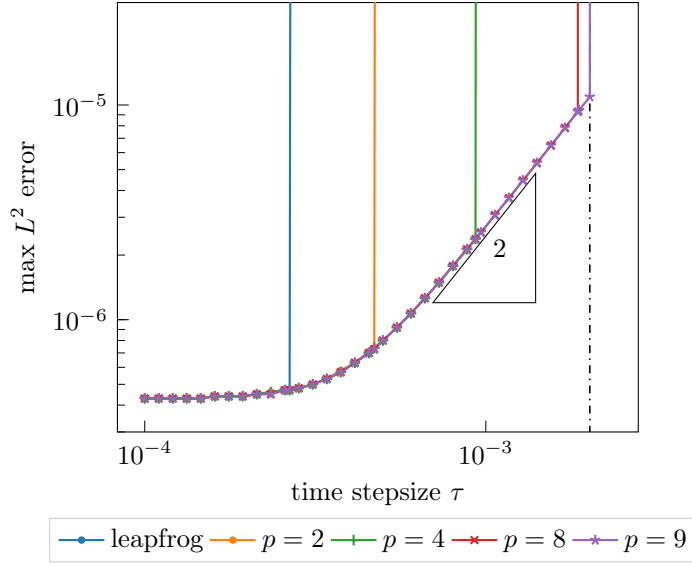

\subsection{Order of convergence}~\label{subsec:convergence-example}
We apply the LTS method~\eqref{eq:lts-scheme} with the LFC filter~\eqref{eq:lfc-polynomials} and the stabilization parameter~\eqref{eq:special-stab-param} for \(\lfcstabparam = 1\) and various values of \(\lfcdegree\). 

In \Cref{fig:lfc-convergence}, we illustrate the stability and convergence behavior with LFC polynomials of degrees \(\lfcdegree = 2, 4, 8, 9\). 
We observe second-order convergence in time until the error reaches a plateau stemming from the spatial discretization. 
Moreover, we see that an increase of \(\lfcdegree\) weakens the \CFL condition  significantly compared to the \leapfrog method (blue). Obviously one can not exceed the maximal stable time stepsize of the \leapfrog method on the coarse part \(\meshLeapfrog\). 

\subsection{Necessity  of stabilization}\label{subsec:stabilization-example}
In this example we show that stabilization, i.e., choosing \(\lfcstabparam>0\), is indispensable in the LFC-LTS method, see also \cite{CarH22} for second-order differential equations. For this we consider the following one-dimensional example
\begin{subequations}\label{eq:1D-equation}
    \begin{align}
        \derv{t} \sol[first] &= -\partial_x \sol[second], && \domain\times(0,T), \\
        \derv{t} \sol[second] &= -\partial_x \sol[first], && \domain\times(0,T), \\
        \sol[first](0) &= \sol[first]^0, \quad \sol[second](0) = \sol[second]^0, && \domain, \\
        \sol[first] &= 0, && \partial\domain\times(0,T),
    \end{align}
\end{subequations}
on \(\domain = (0, 1)\). We choose the initial values such that
\begin{align*}
    \sol[first](x, t) &= \sin(2\pi x)\cos(2\pi t) \\
    \sol[second](x, t) &= -\cos(2\pi x)\sin(2\pi t) .
\end{align*}
is the exact solution of \eqref{eq:1D-equation}.

We use a spatial grid where all cells have diameter \(\maxmeshdiam = 0.009975\) except for one cell in the middle of \(\domain\) with diameter \(\minmeshdiam = 0.0025 \approx \maxmeshdiam /4\).
In \Cref{fig:influence-of-stabilization} we show the error of the LFC-LTS method with polynomial degrees \(p=3, 4, 5\) and stabilization~\eqref{eq:special-stab-param} with \(\lfcstabparam=0\) and \(\lfcstabparam=0.1\), on the left and the right picture, respectively.
As we can see, for \(\lfcstabparam=0\), the method becomes unstable for certain time stepsizes whereas the slightly stabilized method with \(\lfcstabparam=0.1\) has no deviations. 
\begin{figure}
    \centering
    \input{tikz/1D-stabilization/LTS-stabilization.tex}
    \caption{Error of the LFC-LTS method for the example in Section~\ref{subsec:stabilization-example} with polynomial degrees \(\lfcdegree = 3, 4, 5\). Left: without stabilization ($\lfcstabparam = 0$), right: with stabilization ($\lfcstabparam = 0.1$).
    The dash-dotted line depicts the maximal stable time-stepsize of the \leapfrog method on the coarse mesh \(\meshLeapfrog\). }
    \label{fig:influence-of-stabilization}
\end{figure}

\subsection{Runtime comparison}  \label{sec:runtime}

Next, we compare the runtimes of the LFC-LTS method, the LI method, and the \leapfrog method at two different examples. 
We revisit~\eqref{eq:TE-system}  on  \(\domain = (0, 4)^2\) and use an equidistant spatial grid with \(\meshdiam \approx 0.022\). For the dG space we use the polynomial degree \(\dGdegree = 2\). 
Cells with center in the ball \(\norm{x} \leq r\) with \(r\in\{0.1, 0.5\}\) are refined twice.
The minimal mesh diameter on the whole mesh is \(\minmeshdiam \approx 0.0055\). 
Moreover we select the time-stepsize \(\tau = 0.0022\) which is small enough to balance time and space discretization errors. For \(r = 0.1\) and \(r= 0.5\) the $\Ltwo$ errors are \(4\cdot 10^{-4}\) and \(7.6\cdot 10^{-4}\), respectively for all three methods.

In \Cref{tab:runtime}, we summarize the runtimes of the \leapfrog, LI, and LFC-LTS method for the two values of $r$. For the \leapfrog method we choose the maximum stable stepsize \(\tau = 6.8\cdot 10^{-4}\). For the LFC-LTS scheme, we choose the LFC polynomial degree \(\lfcdegree = 4\) and  \(\tau = 2.2\cdot 10^{-3}\). 

For $r=0.1$, the two \lti schemes LI and LFC-LTS clearly outperform the \leapfrog method, with the LFC-LTS method being more than twice as fast. In the second example, where $r=0.5$, the percentage of dofs in the refined region is still relatively large, so that the computational cost for solving the linear system in the LI method is significant. Here, the \leapfrog method is slighly faster than the LI method but much slower than the LFC-LTS method. 

\begin{table}
	\begin{center}
		\begin{tabular}{l|r@{}rr|r@{}rr}
			& \multicolumn{3}{c|}{$r=0.1$} & \multicolumn{3}{c}{$r=0.5$}\\
			\hline
			total dofs   & $1\,782\,756$ & & & $2\,097\,522$& & \\
			refined dofs & $15\,012$ &  &\( 0.84\% \)   &    $353\,430$&  &\( 16.85\%\)  \\
			\hline
			runtime \leapfrog     &     $71.1$ & sec  & $100.0\%$ &       $85.1$ & sec & $100.0\%$\\  
			runtime LI            &     $53.7$ & sec  & $75.5\%$ &       $93.0$  & sec & $109.3\%$\\ 
			runtime LFC-LTS       &     $34.1$ & sec  & $48.0\%$ &      $57.5$  &sec & $67.6\%$ \\ 			 			  
		\end{tabular}
	\end{center}
	\caption{Runtime comparison between the \leapfrog, LI and LFC-LTS methods on two different examples explained in Section~\ref{sec:runtime}.}
	\label{tab:runtime}
\end{table}

Overall we see that the LFC-LTS method performs better than the \leapfrog and the LI method in both scenarios.

\section*{Acknowledgments}
We thank Constantin Carle and Benjamin Dörich for inspiring discussions on local time-integration and its error analysis and helpful comments on the manuscript.

\bibliographystyle{abbrv}
\bibliography{references}
\end{document}